\documentclass{amsart}
\usepackage{amsfonts}
\usepackage{amsmath,amscd}
\usepackage{amsthm}
\usepackage{latexsym}
\usepackage{amssymb}
\usepackage{enumerate}
\usepackage{color}
\newtheorem{theo}{Theorem}
\newtheorem{prop}[theo]{Proposition}
\newtheorem{coro}[theo]{Corollary}
\newtheorem{defi}[theo]{Definition}
\newtheorem{lemm}[theo]{Lemma}

\newtheorem*{claim*}{Claim}

\newcommand{\Z}{\mathbb{Z}}
\newcommand{\Q}{\mathbb{Q}}
\newcommand{\R}{\mathbb{R}}
\newcommand{\C}{\mathbb{C}}
\newcommand{\N}{\mathbb{N}}

\newcommand{\e}{\mathrm{e}}
\newcommand{\re}{\mathrm{Re}}
\newcommand{\im}{\mathrm{Im}}
\newcommand{\ie}{{\it{i.$\,$e.\ }}}
\newcommand{\Ok}{E_+}

\newcommand{\Sp}{\mathrm{SpVt}}

\newcommand{\varpsilon}{\varepsilon}
\newcommand{\Log}{\mathrm{L}}

\title[Signed  fundamental domains  for  number  fields]{Twisters and signed  fundamental domains   for  number  fields}

\author{Milton Espinoza}
\email{milton.espinoza@uv.cl}
\address{Instituto de Matem\'aticas,
Facultad de Ciencias,
   Universidad de Valpara\'iso,
   Gran Breta\~na 1091, 3er piso, Valpara\'iso,
   Chile}

\author{Eduardo Friedman}
\email{friedman@uchile.cl}
\address{Departamento de Matem\'atica, Facultad de Ciencias,
 Universidad de Chile,\newline
Casilla 653, Santiago, Chile.}



\thanks{Espinoza is grateful for the generous support of CONICYT BECAS CHILE 74150071 and the Max-Planck-Institut f\"ur Mathematik. Friedman thanks
   Chilean FONDECYT grant 1140537  for support.}

\begin{document}
\maketitle

\begin{abstract}
We give  a signed  fundamental  domain for the action  on
$\R^{r_1}_+\times{\C^*}^{r_2}$ of the totally positive units $\Ok$ of
a number field $k$ of degree $n=r_1+2r_2$ which we assume   is not totally complex.  Here  $r_1$ and $r_2$ denote the number of real  and complex places of $k$ and  $\R_+$ denotes the positive real numbers.
 The signed fundamental domain consists of
   $n$-dimensional $k$-rational   cones $C_\alpha$,  each equipped  with a
    sign   $\mu_\alpha=\pm1$, with the property that the
   net number of  intersections of the cones with any $\Ok$-orbit is
1.

The   cones  $C_\alpha$ and the signs  $\mu_\alpha$  are explicitly constructed from any set of fundamental totally  positive units and a  set of $3^{r_2}$ ``twisters," \ie elements of $k$ whose arguments at the $r_2$ complex places of $k$ are sufficiently varied. Introducing  twisters gives us the right number of generators   for the cones $C_\alpha $ and allows us to make the   $C_\alpha$ turn in a controlled  way around the origin at each complex embedding.

\end{abstract}

\vspace*{6pt}\tableofcontents

\section{Introduction}\label{Introduction}
The usual embedding of a number field $k$  into a Euclidean space  $V$
gives rise to an action    of  the
  units  of $k$ on $V$.  Good  fundamental domains for this action are important in the study of abelian $L$-functions. For totally real fields,  Shintani \cite{Sh1} showed in 1976    the existence of  fundamental domains consisting of a finite number of $k$-rational polyhedral cones, now known as Shintani cones. A few years later,
 in a posthumous and rarely cited  work \cite{Sh2}, Shintani extended this to all number fields.

 Shintani's papers  gave no practical procedure to construct his cones, or  even to estimate how many cones were needed.
In the late 1980's Colmez \cite{Co} showed for totally real fields the existence of certain special subgroups   of the  units for which he could explicitly construct  Shintani-type fundamental domains for the action of this subgroup on $V$. This was a significant theoretical advance, but it was not effective since no practical procedure
is known for producing Colmez's special  units, except in the quadratic or cubic case \cite{DF1}.

A few years ago Charollois, Dasgupta and Greenberg \cite{CDG}, and independently Diaz y Diaz and Friedman \cite{DF2},  found a way around this non-effectiveness for totally real fields by introducing signed fundamental domains.  Espinoza \cite{Es} then found effective signed fundamental domains for number fields having exactly one   complex place.

Signed fundamental domains can be naturally defined if one defines   fundamental domains using characteristic functions. Indeed, a set of subsets $\{C_\alpha\}_{\alpha\in J} $ of a space $Y$, on which a countable group $G$ acts, is a fundamental domain if
$$
\sum_{\alpha\in J}\sum_{\varepsilon\in G}\chi_{C_\alpha}^{\phantom{-1}}(\varepsilon\cdot y)=1\qquad\qquad(\forall y\in Y),
$$
where $\chi_{C_\alpha}^{\phantom{-1}}$ is the characteristic function of $C_\alpha$. A signed fundamental domain  is a finite list   $\{C_\alpha\}_{\alpha\in J}$ of subsets of $ Y $
and a corresponding list of signs $\mu_\alpha=\pm1$ such that
$$
\sum_{\alpha\in J}\mu_\alpha\sum_{\varepsilon\in G}\chi_{C_\alpha}(\varepsilon\cdot y)=1\qquad\qquad(\forall y\in Y).
$$
For technical reasons \cite[Lemma 5]{DF2}, we also require that the cardinality of $C_\alpha\cap(G\cdot y) $ be bounded independently of $y\in Y$.

 Before going into details, we outline the difficulties that appear with Shintani cones when  $k$ is not  totally real. Let $k$ be a number field of degree $n=r_1+2r_2$ having $r_1$ real places and $r_2$ pairs of complex conjugate embeddings. The most obvious problem is that we no longer have $n$ natural generators for the $n$-dimensional Shintani cones. In the totally real case, following Colmez,  we can use $1\in k$ and $n-1$ independent  units of $k$ to generate the cones, but each complex case leaves us a unit short of the $n$ generators that we need. Thus,  we must introduce one new generator for each complex place. An additional difficulty is that $ \R_+^{r_1}\times{\C^*}^{r_2}$, while it is a cone,  is not convex if $r_2>0$. This means that $n$ independent vectors are liable to generate a cone  including  (nonzero) points outside  $ \R_+^{r_1}\times{\C^*}^{r_2}$, for example where one complex coordinate  vanishes.
We choose  the new   generators   for the cones    so as to force all of them inside convex subsets of $ \R_+^{r_1}\times{\C^*}^{r_2}$. These   generators we call twisters, as their arguments at the complex places are chosen so as to twist the generators into a convex conical  sector.

The approach to signed fundamental domains in \cite{DF2} \cite{Es}  and here can be summarized in the following commutative diagram of topological spaces
\begin{equation}\label{doblediagram} 
\begin{CD}
\mathfrak{X}\times\R @>g  >f  > \R_+^{r_1}\times{\C^*}^{r_2}\\
@VV\widehat{\pi}V @VV\pi V\\
\widehat{T} @>G >F > T
\end{CD}  .
\end{equation}
We begin by explaining the    vertical arrow on the right. Suppose we are given totally positive independent units $\varepsilon_1,....,\varepsilon_r$ generating a subgroup $\mathcal{E}$ of finite index in  the units of $k$. Thus $\mathcal{E}$ acts on
$ \R_+^{r_1}\times{\C^*}^{r_2}$, where $\R_+$ denotes the multiplicative group of positive real numbers. Let $T= (\R_+^{r_1}\times{\C^*}^{r_2})/\mathcal{E}$ be the quotient manifold and $\pi:\R_+^{r_1}\times{\C^*}^{r_2}\to T$ the natural  map.

It is well-known that $T$ is homeomorphic to the product  of an $(n-1)$-torus with $\R$. This is made explicit by the
left vertical arrow in \eqref{doblediagram} and the maps $g$ and $G$, in the following sense. The standard (additive) model for $T$ is $\widehat{T}:=(\R^{n-1}/\Lambda)\times\R$, where  $ \Lambda\subset\R^{n-1}$ is a lattice of dimension $n-1$. We let  $\widehat{\pi}:\R^{n-1}\times\R\to \widehat{T}$ be the natural quotient map and $g:\R^{n-1}\times\R\to   \R_+^{r_1}\times{\C^*}^{r_2}$ the group homomorphism (``exponential") which induces a homeomorphism $G:\widehat{T}\to T$   of the quotient manifolds, \ie $ G\circ\widehat{\pi}=\pi\circ g$. Now let $\mathfrak{X}$ be a fundamental domain for the action of $\Lambda$ on $\R^{n-1}$ and restrict $g$ and $\widehat{\pi}$ to $\mathfrak{X}\times\R$. This defines everything in  \eqref{doblediagram}, except the interesting part, namely $f$ and $F$.

We obtain a classical fundamental domain $\mathcal{F}$ (used by Hecke and Landau, for example)  for the action of
$\mathcal{E}$ on $\R_+^{r_1}\times{\C^*}^{r_2}$ by letting $\mathcal{F}:=g(\mathfrak{X}\times\R)$. Unfortunately,  $\mathcal{F}$  is of limited  use  because of its complicated geometry. In particular, it is difficult to describe the   intersection of $\mathcal{F}$ with  fractional ideals of $k$, a vital step  in Shintani's treatment of $L$-functions.

 To remedy this, for totally real fields Colmez  deformed $g$ to a new and simpler function $f$ so that its image  $f(\mathfrak{X}\times\R)$ is a union of $k$-rational polyhedral cones. From our present standpoint, the work of Colmez \cite{Co} can be described as follows.
Take    the fundamental domain $\mathfrak{X}=[0,1]^{n-1}$ for the lattice $\Lambda:=\Z^{n-1}\subset\R^{n-1}$. There is a well-known decomposition,  parametrized by the symmetric group $S_{n-1}$,
$$
\mathfrak{X}=\bigcup_{\alpha\in S_{n-1}}\mathfrak{X}_\alpha
$$
of the hypercube into simplices $\mathfrak{X}_\alpha$ \cite[eq. (19)]{DF2}.
For  each $\alpha\in S_{n-1}$,   let $A_\alpha:\mathfrak{X}_\alpha\to \R_+^{n}$  be the unique affine function which on each vertex $\kappa$ of $\mathfrak{X}_\alpha$ takes the value $g(\kappa\times 0)\in\R_+^{n}$ (recall that Colmez only dealt with the totally real case). Then we can unambiguously define $f:\mathfrak{X}\times\R\to \R_+^{n}$  by
$ 
f(x\times t):=\e^t A_\alpha(x)$ where $x\in\mathfrak{X}_\alpha,\ t\in\R$. 
As  $\R_+^{n}$  is convex,    $f$ takes values in $\R_+^{n}$ and  $f(\mathfrak{X}_\alpha\times\R)$ is the cone  $\overline{C}_\alpha$ generated by the    $g(\kappa\times 0)$ as $\kappa$ ranges over the vertices of $\mathfrak{X}_\alpha$. Colmez  \cite{Co} proved that if the cones  $\overline{C}_\alpha$ meet only along common faces, then   $\bigcup_\alpha   C_\alpha$ is a fundamental domain for the action of $\mathcal{E}$ on $\R_+^{n}$. Here  $C_\alpha$ is $\overline{C}_\alpha$ minus  some boundary faces.

Colmez's proof is rather complicated, but it can be greatly simplified \cite{DF2} by observing that the decomposition  $\mathfrak{X}=\bigcup_\alpha\mathfrak{X}_\alpha$ is good enough that  $f$   induces a map $F:\widehat{T}\to T$ of the quotient manifolds  in \eqref{doblediagram}.  One can then use topological degree theory  to show that $\bigcup_\alpha C_\alpha$ is  a signed fundamental domain with weights $\mu_\alpha$ given by the degree of $f$ restricted to $\mathfrak{X}_\alpha\times\R$. If we add Colmez's hypothesis that the  $\overline{C}_\alpha$ meet only along common faces, then $\mu_\alpha=1$ for all $\alpha\in S_{n-1}$ and the signed fundamental domain is a true one.

When $k$ is not totally real we would like to have a similar construction. The topology is unchanged, as the manifolds $\widehat{T}$ and $T$ in \eqref{doblediagram} are still homeomorphic to the product of an $(n-1)$-torus with $\R$. It is also easy to write down the homomorphism $g$ inducing the homeomorphism $G$ in \eqref{doblediagram}. However, as we remarked above,  the cone $\R_+^{r_1}\times{\C^*}^{r_2}$ is not convex, so to define $f$ we must ensure that the generators of the cones lie in convex subsets of $\R_+^{r_1}\times{\C^*}^{r_2}$.    Moreover, the decomposition $\mathfrak{X}=\bigcup_\alpha\mathfrak{X}_\alpha$ must again allow  $f$ to   descend  to a map $F$ of the quotient manifolds.

An additional difficulty is that the natural choice of $g$ does not give cone generators in $k$, as in general  $g(\kappa\times 0)\notin k$ for $\kappa $ a  vertex of $\mathfrak{X}_\alpha$. We fix this by modifying $g$ slightly. Fortunately, the homotopy involved in formalizing this approximation  does not alter topological degrees, and so this turns out to be a minor difficulty.

Espinoza \cite{Es} obtained a signed fundamental domain   for fields $k$ with exactly one pair of complex embeddings.
In this paper we extend the ideas there  to all non-totally complex number fields. We exclude totally complex fields solely because in this case we have not been able to give a satisfactory description of the boundary faces  that should be included in the signed fundamental domain.    The reader will find in the next section a detailed description of the signed fundamental domain obtained.

As mentioned at the beginning of this paper, signed fundamental domains are useful for working with abelian $L$-functions. More precisely, as in \cite[Cor.\ 6]{DF2} and \cite[Cor. 3]{Es}, using a signed fundamental domain one can explicitly write any  abelian $L$-function as a finite linear combination of Shintani zeta functions \cite{Sh1}. The Shintani functions do not in general satisfy a functional equation of the usual kind, but they do satisfy a ladder of difference equations and are normalized by vanishing integrals \cite[eqs.\ 1.4 and 1.7]{FR}. In a more geometric vein, in a forthcoming doctoral thesis Alex Capu\~nay gives a practical algorithm  producing a true $k$-rational fundamental cone domain starting from a signed one.

\section{Signed fundamental domain}\label{SFD}
We fix a number field $k$ of degree $[k:\Q]=n$, having $r_1$ real embeddings $\tau_1,...,\tau_{r_1}$ and $r_2$ pairs of conjugate complex embeddings 
\begin{align*}
\tau_{r_1+1},\overline{\tau}_{r_1+1},...,\tau_{r_1+r_2},\overline{\tau}_{r_1+r_2},
\end{align*}
which we use to embed $k$ in $\R^{r_1}\times \C^{r_2}$ by $\gamma\to(\gamma^{(i)})_i$ where $\gamma^{(i)}:=\tau_i(\gamma)\ \,(1\le i\le r_1+r_2)$. To save notation, we identify $\gamma\in k$ with its image  $(\gamma^{(i)})_i\in \R^{r_1}\times \C^{r_2}$.
We also fix  independent totally positive units $\varepsilon_1,...,\varepsilon_r$ of $k$, where $r=r_1+r_2-1$,
and assume we have chosen an integer $N_j\ge3\ \, (1\le j\le r_2)$
for each complex embedding. To have the smallest number of cones we should take $N_j=3$ for all $j$.

\subsection{Raising the dimension of a  complex}\label{RaisingDim}
   An ordered $p$-complex $X$ for us will be a decomposition
\begin{equation}\label{DefComplex}
 X:=\bigcup_{\alpha\in  \mathfrak{I}}X_{\alpha,\prec_\alpha},
\end{equation}
where the $  X_{\alpha,\prec_\alpha}\subset V$ are ordered $p$-dimensional simplices in a $p$-dimensional real vector space $V$.  Thus, for each $\alpha$ in the (possibly infinite) index set $\mathfrak{I} $  we are given  the set of vertices  Vt$(X_\alpha)$ of the $p$-simplex $X_\alpha $ and a total ordering $\prec_\alpha$ on Vt$(X_\alpha)$.
We will later want our complexes to satisfy quite a few properties, but the above definition will do  to identify a  complex. We shall be loose with the notation and denote by $X$ both the underlying point set and  the complex, \ie its decomposition \eqref{DefComplex}. Similarly, we shall  often write $X_\alpha$ for both the simplex and the ordered simplex $X_{\alpha,\prec_\alpha}$.

Our purpose in this subsection is to construct a new ordered $(p+1)$-complex
\begin{equation}\label{InductiveY}
Y=Y(X,\omega, [M_1,M_2]):=\bigcup_{\gamma\in J} Y_{\gamma,\prec_\gamma}\subset V\times\R
\end{equation}
from a given  $p$-complex  $X=\bigcup_{\alpha\in  \mathfrak{I}} X_{\alpha,\prec_\alpha}\subset V$, a linear function $\omega:V\to\R$ and an interval $[M_1,M_2]$. Here the $M_i$ are integers and  $M_1<M_2$. In our application $\omega$ will either vanish identically or be determined by the arguments of the units $\varepsilon_i$ at some complex embedding.   The new ambient vector space is $V\times\R$, and the new index set is
\begin{equation}\label{InductiveJ}
J:=\big\{(\alpha,v,\ell)\big|\,\alpha\in  \mathfrak{I},\ v\in \mathrm{Vt}(X_\alpha),\ \ell\in\Z,\ M_1\le\ell< M_2  \big\}.
\end{equation}
If $\mathfrak{I}$ is a finite set, then so is $J$ and its cardinality   is
\begin{equation}\label{InductiveCard}
|J|=(M_2-M_1)(p+1)| \mathfrak{I}|.
\end{equation}

To define the simplex $Y_\gamma$   for $\gamma=(\alpha,v,\ell)\in J$ requires  some preliminaries. Let $A:V\to(0,1]$ be the  upper fractional part of $\omega$, \ie $A(u)\in\R$ is uniquely determined by
\begin{equation}\label{Avchar0}
 A(u)-\omega(u) \in\Z,\qquad0<A(u) \le 1 \qquad\qquad(u\in V).
  \end{equation}
 We use $A$ to define a new  total order $\prec_\alpha^A$ on  $\mathrm{Vt}(X_\alpha)$ by
  \begin{equation}\label{SalphaA0}
u\prec_\alpha^A u^\prime\ \Longleftrightarrow\
\Big[
\big[A(u)<A(u^\prime)\big]\ \ \mathrm{or}\ \ \big[A(u)=A(u^\prime)\ \mathrm{and}\  u \prec_\alpha u^\prime\big] \Big] 
  \end{equation}
for $u,u^\prime\in  \mathrm{Vt}(X_\alpha)$. 
Note that if $\omega=0$ identically, then   $ \prec_\alpha^A\ =\ \prec_\alpha$.

We  now define $ Y_\gamma $ in \eqref{InductiveY} for $\gamma=(\alpha,v,\ell)\in J$. Order the vertices $\{v_0,v_1,...,v_p\}$ of  $ X_\alpha$ so that
$  v_0\prec_\alpha^A v_1\prec_\alpha^A\cdots \prec_\alpha^A v_p.$ Since $v\in \mathrm{Vt}(X_\alpha)$ by definition \eqref{InductiveJ}, there is a unique $j=j_\gamma$ such that $v=v_j$. Note that $j$ depends on the order $\prec_\alpha^A$. Let $[s_1,...,s_t] $ denote the  convex hull of   $s_1,...,s_t$.
Then $Y_\gamma\subset V\times\R$ is defined as
\begin{align}
\label{Newvertices}Y_\gamma:=& \big[v_0\times\big(A(v_0)-\omega(v_0)+\ell \big),\ v_1\times\big(A(v_1)-\omega(v_1)+\ell \big), ..., \\ &v_j\times\big( A(v_j)-\omega(v_j)+\ell \big), v_j\times\big(A(v_j)-\omega(v_j) +\ell-1\big),\nonumber \\ & v_{j+1}\times\big(A(v_{j+1}\big)-\omega(v_{j+1}) +\ell-1\big),...,v_p\times\big(A(v_p)-\omega(v_p) +\ell-1\big)\big].\nonumber
\end{align}
Note that   vertices   of $Y_\gamma$ map  to
  vertices of $X_\alpha$  by the projection $\pi_V:V\times\R\to V$, and that the coordinate we have added to each vertex is an integer $\big($see \eqref{Avchar0}$\big)$.

Finally,  the new ordering $\prec_\gamma$ on the vertices of the new simplex  $ Y_\gamma$ is defined by
\begin{align}
&\label{gammaOrder}\rho \prec_\gamma \rho^\prime \\ &\Longleftrightarrow 
\Big[
\big[ \pi_V(\rho)\prec_\alpha   \pi_V(\rho^\prime)\big]\ \ \mathrm{or}\ \ \big[\pi_V(\rho)=   \pi_V(\rho^\prime)\ \mathrm{and}\ \pi_\R(\rho^\prime) <   \pi_\R(\rho) \big] \Big],\nonumber
  \end{align}
where $\pi_\R:V\times\R\to\R$ is the projection onto the second factor.
Note that   $\prec_\gamma$ is defined using the original   order $\prec_\alpha$ on the vertices of $X_\alpha$, even though we used  the modified order $\prec_\alpha^A$ to construct $Y_\gamma$.

\subsection{The $(n-1)$-complex $\mathfrak{X}$}\label{ComplexX}   We start from the  trivial $0$-complex
\begin{equation}\label{X0}
X_0:=X_{O,\prec_O}\subset\R^0,
  \end{equation}
where $\R^0:=\{0\}$ is the trivial vector space, the index set is $\{O\}$ (or any one-element set), and  the ordered 0-simplex $X_{O,\prec_O}=\{0\}$ is equipped with the trivial (empty)  ordering $\prec_O$ on the single vertex $0$.
 For $1\le j\le r:=r_1+r_2-1$, let $\omega_{j-1}:\R^{j-1}\to\R$ be the trivial linear function $\omega_{j-1}(x):=0$, and define
\begin{equation*}
X_j:=Y(X_{j-1},\omega_{j-1},[0,1])\subset\R^j\qquad\qquad(1\le j\le r),
  \end{equation*}
where $Y$ was defined in \eqref{InductiveY}.\footnote{\ Although we will not need this, $X_r$  is  the well-known decomposition of the $r$-cube $[0,1]^r$ into $r!$ simplices determined by the order of the coordinates \cite[eq.\ (19)]{DF2}.}
If $k$ is totally real, so $r=n-1$,  we are done constructing $\mathfrak{X}:= X_{n-1}$. Otherwise, we must carry out  further  steps, one for each complex place of $k$. For $1\le j\le r_2$, let $\omega_{r+j-1}:\R^{r+j-1}\to\R$   be given by
\begin{equation*}
\omega_{r+j-1}(x):=\frac{N_j}{2\pi}\sum_{\ell=1}^r x[\ell]\arg\!\big( \varpsilon_\ell^{(r_1+j)}\big) \qquad\big( x=(x[1],x[2],...,x[r+j-1]) \big),
  \end{equation*}
where we recall that for each $1\le j\le r_2$ we have fixed an integer $N_j\ge3$.
For the sake of definiteness, the arguments above are taken so that
$-\pi<  \arg(z) \le\pi$,
but   any branch  would do as well.
Define
\begin{equation}\label{rjcomplexdefined}
X_{r+j}:=Y(X_{r+j-1},\omega_{r+j-1},[0,N_j])\subset\R^{r+j} \qquad\qquad(1\le j\le r_2),
  \end{equation}
 with $Y$ as in \S\ref{RaisingDim}.

We have thus constructed an $(n-1)$ ordered complex
\begin{equation}\label{neededcomplex}
\mathfrak{X}=\bigcup_{\alpha\in\mathfrak{I}} \mathfrak{X}_\alpha:=X_{n-1}\subset \R^{n-1} \qquad\qquad  \Big( | \mathfrak{I}|=(n-1)!\prod_{j=1}^{r_2}N_j\Big),
  \end{equation}
 corresponding to $j=r_2$ above.
We have dropped the order $\prec_\alpha$ on $\mathrm{Vt}(\mathfrak{X}_\alpha)$ from our notation since, once $\mathfrak{X}$ is constructed  we will have no further interest in  ordering the set of  vertices
\begin{equation}\label{neededsimplices}
\mathrm{Vt}(\mathfrak{X}_\alpha):=\{v_{0,\alpha},v_{1,\alpha},...,v_{n-1,\alpha}\}\subset\Z^{n-1}.
  \end{equation}
However,  we will use the fact the vertices  $ v_{\ell,\alpha}$ have integral coordinates. This is is clear since we started from the single vertex $0$ of $X_0$. As we  noted after \eqref{Newvertices}, in passing from $X_{j-1}$ to $X_j$ the new coordinate appended to a vertex   is always integral.

\subsection{The twister function $\beta$}\label{Twistersdefined}
Recall that we have fixed  $r:=r_1+r_2-1$ independent totally positive units $\varepsilon_1,...,\varepsilon_r$ in the number field $k$ and integers $N_j\ge3$ for $1\le j\le r_2$.  We must now also fix ``twisters" $\beta(x)$, which are totally positive elements of $k^*$ whose arguments at the various complex places are sufficiently close to those of certain roots of unity determined by $x$. More precisely, choose a function $ \beta:\Z^{n-1}\to k^*$ with the following properties:

 \noindent$\bullet$  $\beta(x)$ is a  totally positive element of $k$.

 \noindent$\bullet$ For   $x=(x[1],...,x[n-1])\in\Z^{n-1}$ and      $1\le j\le r_2$, there is a $t_j(x)\in\R$ such that
\begin{equation}\label{beta0}
\frac{|t_j(x)|}{N_j}<\frac{1}{4} -\frac{1}{2N_j} ,\quad \frac{\beta(x)^{(j+r_1)}}{|\beta(x)^{(j+r_1)}|}=\exp\!\Big( 2\pi i\big(x[j+r]+t_j(x) \big) /N_j\Big).
\end{equation}
For example, if $N_j=3$ this requires that,  at the $(j+r_1)$-th embedding, the argument of $\beta(x)$ be no farther than $\pi/6$ from that of   $\exp(2\pi ix[j+r]/3)$.

\noindent$\bullet$ We have
 \begin{equation}\label{periodic}
  \beta(x+ \lambda )=\beta(x)\qquad(x\in\Z^{n-1},\ \lambda \in \Lambda :=\Z^r\times N_1\Z\times\cdots\times N_{r_2}\Z).
 \end{equation}
 Thus, the twister   $\beta$ amounts to a function on the  finite set  $\Z^{n-1}/ \Lambda$.       Note that  in \eqref{beta0} we can take $t_j(x+\lambda )=t_j(x) $.
Twister functions exist by our assumption  $N_j\ge 3 $ and the density of  $k$  in $\R^{r_1}\times\C^{r_2}$.\footnote{\ $\beta(x)$ can be computed by taking a $\Q$-basis $\gamma_1,...,\gamma_n$ of $k\subset \R^{r_1}\times\C^{r_2}$, finding for each $x\in\Z^{n-1}$ with $ 0\le x[j+r_1]<N_j  \ \,(1\le j\le r_2 )$, rational solutions $a_\ell=a_\ell(x)\ \,(1\le\ell\le n)$  to the approximate equality
 $$
\big(1,...,1,\exp(2\pi i x[r_1+1]/N_1),...,\exp(2\pi i x[n-1]/N_{r_2})\big)\ \approx\ \sum_{\ell=1}^n a_\ell \gamma_\ell,
$$
 and letting $\beta(x):=\sum_{\ell=1}^n a_\ell \gamma_\ell$. Then extend $\beta$  to all $x\in\Z^{n-1}$ by periodicity  \eqref{periodic}.}

\subsection{Signs} We  need to compute  three determinants to define the sign
\begin{equation}\label{mualphadefined}
\mu_\alpha:=(-1)^{r_2(r_2-1)/2}\,\mathrm{sign}\big(\!\det(R)\det(V_\alpha)\det(W_\alpha)\big).
\end{equation}
Although $R$ and $V_\alpha$ will prove to be invertible matrices,  $W_\alpha$ may not be. Thus,  $\mu_\alpha$ may take the values $\pm 1 $ or 0.  

Recall that we have listed the embeddings $\tau_i$ of $k$ so that they are real for $1\le i\le r_1$ and  are complex conjugate pairs $\tau_i,\bar\tau_i$ for $r_1<i\le r_1+r_2$. The $(r_1+r_2)\times (r_1+r_2)$ matrix $R=\big( r_{i j}\big)$ is
defined by
\begin{equation}\label{Reg1}
  r_{i j}  :=\begin{cases}
 1 &\ \mathrm{if}\ i=1,\\
 \log|\varepsilon_{i-1}^{(j)}| &\ \mathrm{if}\ 2\le i\le r_1+r_2.
 \end{cases}\quad\quad   (1\le i,j\le r_1+r_2).
\end{equation}
Actually, $|\!\det(R)|=  2^{-r_2}n\,\text{Reg}(\varepsilon_1,\ldots,\varepsilon_r)$, where Reg is the regulator of the units $\varepsilon_i$, but we   care here  only for the sign of $\det(R)$.

Define $V_\alpha$ as the  $(n-1)\times(n-1)$ matrix whose $i$-th column   is  $v_{i,\alpha}-v_{0,\alpha}\in\Z^{n-1}\ \,(1\le i\le n-1)$. Here $\mathrm{Vt}(\mathfrak{X}_\alpha)=\{ v_{0,\alpha},v_{1,\alpha},\ldots, v_{n-1,\alpha} \}\subset\Z^{n-1}$ was given in \eqref{neededcomplex} with the complex $\mathfrak{X}$.
To define the matrix $W_\alpha$ in \eqref{mualphadefined},   let
\begin{align}\label{generator}
w_\ell=w_{\ell,\alpha}:=\beta(v_{\ell,\alpha})&\prod_{j=1}^{r}\varepsilon_j^{v_{\ell,\alpha}[j]} \qquad
(0\le\ell\le n-1).
\end{align}
Thus $w_\ell \in k^*$.
 Define $W_\alpha$ as the real $n\times n$ matrix with $(\ell+1)$-th row 
\begin{align*}
&\tau_1(w_\ell),\tau_2(w_\ell),...,\tau_{r_1}(w_\ell), \\
&\qquad\qquad\qquad \re\big(\tau_{r_1+1}(w_\ell)\big),\im\big(\tau_{r_1+1}(w_\ell)\big),...,\re\big(\tau_{r_1+r_2}(w_\ell)\big),\im\!\big(\tau_{r_1+r_2}(w_\ell)\big)
\end{align*}
for $0\le \ell\le n-1$.

\subsection{Cones}
For $\alpha\in  \mathfrak{I}$ as in \eqref{neededcomplex}, let
\begin{equation}\label{closedcone}
 \overline{C}_\alpha:= \Big\{\sum_{\ell=0}^{n-1}x_\ell w_{\ell,\alpha}\big|\,\, x_\ell\ge0 \text{ for }0\le\ell\le n-1\Big\}-\{0\},
\end{equation}
where $ w_{\ell,\alpha}\in k\subset\R^{r_1}\times\C^{r_2}$ was defined in \eqref{generator}.\footnote{\ Note that we removed the origin in \eqref{closedcone}. A great part of our efforts will be directed to showing that $\overline{C}_\alpha$ is contained in $\R_+^{r_1}\times{\C^*}^{r_2}$. Here the difficulty is in ensuring that the complex components do not vanish.}  
If $w_{0,\alpha},...,w_{n-1,\alpha}$ is not an $\R$-basis of the real vector space $\R^{r_1}\times\C^{r_2}$, \ie if $\mu_\alpha=0$, we will  not   define  the subset $C_\alpha$ of $\overline{C}_\alpha$. If $\mu_\alpha\not=0$,  we can write $e_1=(1,0,...,0)\in \R^{r_1}\times\C^{r_2}$ uniquely as
$ e_1 =\sum_{\ell=0}^{n-1} y_{\ell,\alpha} w_{\ell,\alpha}, $ where $y_{\ell,\alpha}\in\R$.
We shall prove (see Lemma \ref{e1})  that
 $y_{\ell,\alpha} \not=0$ for $0\le\ell\le n-1$.
Define the cones
\begin{equation}\label{cone}
 C_\alpha:= \Big\{\sum_{\ell=0}^{n-1}x_\ell w_{\ell,\alpha}\big|\, x_\ell\in R_{\ell,\alpha}   \Big\},
\quad
R_{\ell,\alpha}:=\begin{cases}
[0,\infty)\  &\text{if } y_{\ell,\alpha}>0,\\
(0,\infty)\  &\text{if } y_{\ell,\alpha}<0
\end{cases}.
\end{equation}
Note that  $C_\alpha $ is the open $n$-dimensional cone generated by the $w_{\ell,\alpha}$, together with some of its boundary faces.

 We can now state our  main result.
\begin{theo}\label{Maintheorem} Let independent totally positive units $\varepsilon_1,...,\varepsilon_r$ and  a  twister funtion  $\beta$  be given as in \S$\mathrm{\ref{Twistersdefined}}$ for a number field $k$,  let $E:=\langle \varepsilon_1,...,\varepsilon_r\rangle$ be the subgroup of the units of $k$ generated by the $\varepsilon_\ell$, and assume that $k$ is not totally complex.   Then    $\{C_\alpha,\mu_\alpha\}_{\substack{\alpha\in  \mathfrak{I}\\ \mu_\alpha\not=0}}$, defined in  \eqref{cone},  \eqref{mualphadefined} and \eqref{neededcomplex}, is a signed fundamental domain for the action of $E$ on  $\R_+^{r_1}\times{\C^*}^{r_2}$ consisting of $k$-rational signed cones.

 Thus, the cones  $C_\alpha$ have generators $w_{\ell,\alpha}\in k$   defined in \eqref{generator}, and for any $x\in \R_+^{r_1}\times{\C^*}^{r_2}$ we have
\begin{equation}\label{BasicCount}
\sum_{\substack{\alpha\in  \mathfrak{I}\\ \mu_\alpha\not=0}}\mu_\alpha\sum_{\varepsilon\in E}\chi_{C_\alpha}^{\phantom{-1}}(\varepsilon x)=1,
\end{equation}
where  $\chi_{C_\alpha}^{\phantom{-1}}$ is the characteristic function of $C_\alpha$. Furthermore, $\chi_{C_\alpha}^{\phantom{-1}}(\varepsilon x)=0$ except for $\varepsilon$ in a finite set  of cardinality bounded independently of $x$.
\end{theo}
We recall that in \eqref{neededcomplex} there is a  free choice of integers $N_j\ge3$ for $1\le j\le r_2$ and that the number of  cones $C_\alpha$ is at most $(n-1)!\prod_j N_j$, where $n=[k:\Q]$. If we pick all $N_j=3$,  then there are at most $3^{r_2}\cdot (n-1)!$  cones.

If $k$ is totally complex, we still prove \eqref{BasicCount}, but only for $x$ outside the $E$-orbit of the boundary of all $\overline{C}_\alpha$. The excluded  set has Lebesgue measure 0, but is still unfortunate for the calculation of abelian $L$-functions.

There are two very different   parts to the proof of Theorem \ref{Maintheorem}. The  first   (see \S\ref{LSC})    consists of    showing that if a complex $X$  has a number of  properties with respect to a lattice $\Lambda  $, then so does the  complex $ Y(X,\omega,[M_1,M_2])$   with respect to the lattice $ \Lambda\times (M_2-M_1)\Z $.  This allows us to construct inductively an $(n-1)$-complex $\mathfrak{X}$ and a function $f:\mathfrak{X}\times\R\to \R_+^{r_1}\times{\C^*}^{r_2} $ whose image gives the cones in the signed fundamental domain. The second part     of the proof (see  \S\ref{topology}) is mainly a calculation of certain global and local topological degrees associated to $f$. As in   \cite{DF2} \cite{Es}, degrees enter  because     Theorem \ref{Maintheorem} can be interpreted as an instance of the local-global principle on suitable manifolds.

\section{Lattice-adapted ordered simplicial complexes}\label{LSC}
\subsection{Affine preliminaries}\label{AffinePreliminaries}
Let $V$ be a real vector space and  $\{v_0,...,v_p\}\subset V$ a finite subset. It is called affinely independent if for a fixed $j$ the set  $\{v_i-v_j\}_{\substack{0\le i\le p\\ i\not=j}} $ is   linearly independent. This notion does not depend on the choice of $j$. If $p=0$, any $v_0\in V$ is affinely independent.

The convex hull $S$ of  $\{v_0,...,v_p\}$ is
\begin{equation}\label{Simplex}
S= [v_0,v_1,\ldots,v_p]:=\bigg\{\omega =\sum_{j=0}^p t_j v_j\Big|\,  t_j\ge0,\ \sum_{j=0}^p t_j =1 \bigg\}.
\end{equation}
We   call $S$ a  simplex  if the $ v_j $ are  affinely independent. Then Vt$(S):=\{  v_0,v_1,\ldots,v_p\}$, its set      of vertices, is uniquely determined by the point set $S$.   If we wish to note the dimension of   $S $, we call it a $p$-simplex. An $\ell$-face, or simply a face, of    $S$ is an $\ell$-simplex $K$ such that $ \mathrm{Vt}(K)\subset \mathrm{Vt}(S)$.

The $t_j$ in \eqref{Simplex} are called the barycentric coordinates of $\omega $
   and $ \sum_{j=0}^p t_j v_j$ is called its barycentric expansion  (with respect to $S$).
   If $S$ is a simplex we write $\Sp(\omega)=\Sp(\omega;S)$ for the set of spanning vertices of $\omega$, \ie those $v_j\in \mathrm{Vt}(S)$ with $t_j>0$ in \eqref{Simplex}. Note that if $\omega\in K\subset S$, where $K$ is a face of the simplex $S$, then
\begin{equation}\label{spaninface}
\Sp(\omega;S)=\Sp(\omega;K).
\end{equation}

   If   $U$  is a real vector space, a map $A:V\to U$ is  affine if $A(v )=L(v )+C$, where $L:V\to U$ is
   $\R$-linear and $C\in U$ is fixed. If $S= [v_0,v_1,\ldots,v_p] \subset V$ is a simplex, a map $H:S\to U$ is called affine if it is the restriction to $S$ of an affine map $A:\R S\to U$. Here $\R S$ is the vector subspace of $V$ spanned by the $v_i$. In terms of the barycentric expansion such a map satisfies
\begin{equation}\label{affinemap}
H(\omega )=\sum_{j=0}^p t_j H(v_j) \qquad \Big(\omega =\sum_{j=0}^p t_j v_j,\quad \sum_{j=0}^p t_j=1,\quad t_j\ge0\ \text{for}\ 0\le j\le p\Big),
\end{equation}
and so is determined by the $H(v_j)$. In fact, $H(\omega )$ is determined by the values of $H$ on the spanning vertices of $\omega$. Conversely, if for each vertex $v_j$ of $S$ we choose some $H(v_j)\in U$, then \eqref{affinemap} defines a
unique affine map $H:S\to U$.

 If
   $S^\prime=[u_0,u_1 ,\ldots,u_\ell]\subset U$ is a simplex in $U$,   a
      map $T:S\to S^\prime$ is called simplicial if it is an affine map from $S$ to $U$ that takes vertices of $S$ to vertices  of $S^\prime$. An injective simplicial map $T$ preserves barycentric coordinates, \ie if $t_j$ is the barycentric coordinate of $\omega\in S$ corresponding to a vertex $v_j$, then $t_j$ is  also the barycentric coordinate of $T(\omega)\in S^\prime$ corresponding to the vertex $T(v_j)$.

\subsection{$\Lambda$-complexes}\label{LambdaComplexsection}
Recall that in \S\ref{RaisingDim} we defined an ordered complex $$X=\bigcup_{\alpha\in\mathfrak{I}} X_{\alpha}.$$ A map of complexes $ X\to X^\prime$, where  $X^\prime=\bigcup_{\beta\in\mathfrak{I}^\prime} X^\prime_{\beta}$, is a map of index sets $\widetilde{T}:\mathfrak{I}\to\mathfrak{I}^\prime$ together with a map of point sets $T:X\to X^\prime$ such that $T$ restricted to each $X_\alpha$ is a simplicial map to $X'_{\widetilde{T}(\alpha)}$. If $\widetilde{T}$  and $T$  are bijections,  the set-theoretic inverse of $T$ is also a map of complexes, and so    the  complexes are isomorphic.

\begin{defi}\label{LambdaComplex}
Suppose $V$ is a  $p$-dimensional real vector space,  $\Lambda\subset V$ is a full lattice   (\ie  a discrete subgroup of $V$ whose $\R$-span is $V$), and
$X=\bigcup_{\alpha\in\mathfrak{I}}X_{\alpha,\prec_\alpha}\subset V$  is an ordered $p$-complex in $V$.  We shall say that $X$ is  a $\Lambda$-complex if it satisfies the following five properties.
\vskip.3cm
\noindent$(i)$ $X$ is a simplicial complex, \ie   for  $\alpha,\beta\in \mathfrak{I}$, $X_\alpha\cap X_\beta$ is   empty or
\begin{equation}\label{Intersects}
X_{\alpha}\cap X_{\beta}=[v_1,\ldots,v_\ell],\quad \text{where} \quad \{v_1,\ldots,v_\ell\}:=\mathrm{Vt}(X_\alpha)\cap\mathrm{Vt}(X_\beta).
\end{equation}
In other words, simplices intersect along common faces.

\noindent$(ii)$  The orders are compatible, \ie 
\begin{equation}\label{Ordered}
v\prec_\alpha w \quad\Longleftrightarrow \quad v\prec_\beta w\qquad\big(   \alpha,\beta\in \mathfrak{I},\ v,w\in \mathrm{Vt}(X_\alpha)\cap \mathrm{Vt}(X_\beta)\big).
\end{equation}
\noindent$(iii)$  The orders are    $ \Lambda$-invariant,   \ie  for   $\alpha,\,\alpha^\prime\in \mathfrak{I}$, if    $v,w\in \mathrm{Vt}(X_\alpha)$ and $v+\lambda,w +\lambda \in \mathrm{Vt}(X_{\alpha^\prime})$ for some $\lambda\in\Lambda$, then
\begin{equation}\label{translation}
v\prec_{\alpha}w\quad \Longleftrightarrow\quad v+\lambda\prec_{\alpha^\prime}w +\lambda .
  \end{equation}
\noindent$(iv)$   $X$ is nearly a fundamental domain for $\Lambda$, \ie  the restriction to $X$ of the  natural quotient map from $V$ to $V/\Lambda$ is surjective, and it is injective when res\-tricted to  the union $\bigcup_{\alpha\in \mathfrak{I}} \stackrel{\circ}{X}_\alpha$ of the interiors  of the  $X_\alpha$.
\vskip.3cm
\noindent$(v)$    The spanning vertices are   $ \Lambda$-equivariant,\footnote{\ As the complex $X=\bigcup_{\alpha\in\mathfrak{I}}X_\alpha$ is assumed simplicial, \eqref{spaninface} shows that the set $\Sp(x)=\Sp(x;X_\alpha)$ of spanning vertices of $x\in X$ is independent of the simplex  $X_\alpha$ containing $x$ used to calculate the barycentric expansion of $x$. We  will  write $ \Sp(x;X)$ when we wish to specify the simplicial  complex involved.} \ie
  \begin{align*}
 \Big[x,   x^\prime\in X,\  \lambda\in\Lambda,\ x^\prime=x+\lambda \Big]\quad\Longrightarrow\quad \Sp(x^\prime)=\Sp(x)+\lambda.
   \end{align*}
\end{defi}
\noindent Note that if  $X$ is simplicial (in the sense of (i) above) and $X^\prime$ is an isomorphic complex, then $X^\prime$  is also simplicial.

  Next we state the main result of this section.
\begin{prop}\label{MainInduction}
Let the $p$-complex  $X=\bigcup_{\alpha\in\mathfrak{I}}X_{\alpha,\prec_\alpha} $  be a $\Lambda$-complex in $V =\R\Lambda$, let $M_1<M_2$ be integers,  let $\widehat{\Lambda}:= \Lambda\times (M_2-M_1)\Z$,  and  let $\omega:V\to \R$ be a linear function. Then the ordered $(p+1)$-complex defined in \S$\mathrm{\ref{RaisingDim}}$,
$$
Y(X,\omega,[M_1,M_2])=\bigcup_{\gamma\in J} Y_{\gamma,\prec_\gamma}\subset V\times\R,
$$
 is a $\widehat{\Lambda}$-complex. Furthermore, if  $\gamma=(\alpha,v,\ell)\in J$ and  we define $\Omega:V\times\R\to\R$  by $\Omega(s\times t):=\omega(s)+t$, then for all   $\kappa\in  Y_\gamma $ we have
\begin{equation}\label{Tineq}
A(v)+\ell-1\le\Omega(\kappa)\le A(v)+\ell,
\end{equation}
where $A(v) $ is the upper fractional part of $\omega(v)$ defined in \eqref{Avchar0}.
\end{prop}
\noindent In the next three subsections we prove a series of lemmas leading to a proof  in  \S\ref{proofofprop} of the above proposition.
\subsection{Simplicial decomposition of $X\times[M_1,M_2]$}\label{ConstSimp}
\begin{lemm}\label{SimplexPiling} Let $I=[0,1]\subset\R$ be the unit interval, let $  S= S_\prec$   be an ordered $p$-simplex in some real $p$-dimensional vector space $V$, and write $S=[v_0,v_1,\ldots,v_p]$ as the convex hull of its $p+1$ vertices, ordered so that  $v_0\prec v_1\prec\cdots\prec v_p$.  For any vertex  $v=v_j\in \mathrm{Vt}(S)$,  let
\begin{align}\label{StimesI}
 S^v&=S_\prec^v:= [v_0\times 1,v_1\times 1,\ldots,v_j\times 1,v_j\times 0,v_{j+1}\times 0,\ldots,v_p\times 0] \\ \subset
 &S\times I\subset V\times\R. \nonumber
\end{align}
Then $S^v$ is an ordered $(p+1)$-simplex if we define, for $w,w^\prime\in\mathrm{Vt}(S^v)$,
\begin{align}\label{neworder}
&w\prec^v w^\prime\Longleftrightarrow\\
&\Big[ [\pi_V(w)\prec \pi_V(w^\prime)]\ \text{or}\  [\pi_V(w)= \pi_V(w^\prime)\ \text{and}\ \pi_\R(w^\prime)<\pi_\R(w)] \Big] ,\nonumber
 \end{align}
where  $\pi_V: V\times\R\to V$ and $\pi_\R: V\times\R\to \R$ are the natural projections. Furthermore,
$ \displaystyle S\times I=\bigcup_{v\in \mathrm{Vt}(S)} S^v,
$ 
 and this  is an ordered simplicial $(p+1)$-complex, \ie it satisfies \eqref{Intersects} and \eqref{Ordered}.
\end{lemm}
\begin{proof} It is immediate that $S^v$ is a $(p+1)$-simplex and that \eqref{neworder} makes $\bigcup_v S^v$ into an ordered complex satisfying \eqref{Ordered}. To see that it is simplicial, \ie satisfies \eqref{Intersects}, we will show that it is isomorphic to the standard simplicial decomposition of $\Delta_p\times I$, where $\Delta_p$ is the standard   $p$-simplex. Indeed, let $e_1,...,e_p$ be the standard basis of $\R^p$ and set $E_0:=0_{\R^p}\in\R^p,\ E_i:=e_i+E_{i-1}\ \,(1\le i\le p)$, and let $A:V \to\R^p $ be the unique affine isomorphism satisfying $A(v_i):=E_{p-i}\ \,(0\le i\le p)$, so that
\begin{align*}
A(S)=[E_p,E_{p-1},...,E_0]
=:\Delta_p=\big\{x\in\R^p\big|\,1\ge x_1 \ge x_2\ge \cdots \ge x_p\ge0\big\}.
\end{align*}
Let $\widehat{A}:V\times\R\to\R^p\times\R$ be the affine isomorphism given by  $\widehat{A}(v\times y):=A(v)\times y\in\R^p\times\R=\R^{p+1}$.
Then  $\widehat{A}(S\times I)=\Delta_p\times I$, and
\begin{align*}
&\widehat{A}(S^{v_j})=[E_p\times1,...,E_{p-j}\times1,E_{p-j}\times0,E_{p-j-1}\times0,...,E_0\times0]
\\ & =\big\{x\in\R^{p+1}\big|\,1\ge x_1 \ge \cdots \ge x_{p-j}\ge x_{p+1}\ge x_{p-j+1}\ge \cdots \ge x_p\ge0\big\}
\end{align*}
$(0\le j\le p)$. This is  the standard decomposition of the product $\Delta_p\times I$,  easily seen to be simplicial.
 \end{proof}

Next we record some simple properties of the above decomposition of $S\times I$.

\noindent$(a)$ The projection $\pi_V: V\times\R\to V$   maps $S^v$ onto $S$, and maps the vertices of $S^v$ bijectively onto the vertices of $S$, except for the vertices $v\times0$ and $v\times1$, both of which map to $v$.

\noindent$(b)$   Let  $ {S^v}^-$ and   ${S^v}^+$ be the $p$-faces of  $S^v$ defined by
\begin{align}\label{RoofFloor}
{S_\prec^v}^- ={S^v}^- :=  & [v_0\times 1,v_1\times 1,\ldots,v_{j-1}\times 1,v_j\times 0,v_{j+1}\times 0,\ldots,v_p\times 0],\\
{S_\prec^v}^+ ={S^v}^+ := & [v_0\times 1,v_1\times 1,\ldots,v_{j-1}\times 1,v_j\times 1,v_{j+1}\times 0,\ldots,v_p\times 0],\nonumber
\end{align}
If   $w\in S$, then $w\times0$ is contained in a proper face of $ S^{v_0}$,  namely in
$    {S^{v_0}}^-:=[v_0\times 0,v_1\times 0,\ldots, v_p\times 0]\subset S^{v_0}.$
Similarly,
  $w\times1$ is contained in a proper face of $S^{v_p}$,
\begin{equation}\label{TopBottom}
  {S^{v_p}}^+:=[v_0\times 1,v_1\times 1,\ldots, v_p\times 1]={S^{v_0}}^-\ +\ (0_V\times1)\subset S^{v_p},
 \end{equation}
  where $0_V\in V$ is the origin in $V$.

\begin{lemm}\label{ComplexPiling} Suppose $X =\bigcup_{\alpha\in \mathfrak{I}} X_{\alpha,\prec_\alpha}  $ is  an ordered simplicial   $p$-complex. Then   
\begin{equation}\label{SalphaI}
X\times I=\bigcup_{\alpha\in \mathfrak{I}}\bigcup_{v\in\mathrm{Vt}(X_\alpha)} X^v_{\alpha,\prec_\alpha^v}
\end{equation}
 is an ordered simplicial $(p+1)$-complex. Here  $I$ is the closed interval $[0,1]$ and $X^v_{\alpha, \prec_\alpha^v}$  was defined in Lemma \ref{SimplexPiling}, taking  $S_\prec:=X_{\alpha,  \prec_\alpha}$.
\end{lemm}

\begin{proof}
The equality of sets in \eqref{SalphaI} is clear from Lemma \ref{SimplexPiling}. By the same lemma, it is easy to see that \eqref{SalphaI} is an ordered $(p+1)$-complex. We now show that the decomposition is simplicial. Suppose 
$  \delta=\omega\times y\in X_\alpha^v\cap X_\beta^w.$ 
We will show that the spanning verticies of $\delta$ satisfy
$\Sp(\delta;X_\alpha^v)=\Sp(\delta; X_\beta^w)$. This suffices as it shows that $\delta$ lies in the convex hull of $ \mathrm{Vt}(X_\alpha^v)\cap\mathrm{Vt}(X_\beta^w) $. Let $L:=X_\alpha\cap X_\beta$, a non-empty   simplex as $\omega\in L$, and a  common face of the simplices $X_\alpha$ and $X_\beta$. Then $\delta\in L\times I$, so by Lemma  \ref{SimplexPiling}, $\delta\in L^\mu $ for some vertex  $\mu\in  \mathrm{Vt}(L)=\mathrm{Vt}(X_\alpha)\cap\mathrm{Vt}(X_\beta) $. But $L^\mu$ is a face of $X_\alpha^\mu$ and of $X_\beta^\mu$, as follows immediately from the construction of Lemma \ref{SimplexPiling}. Therefore $\delta\in X_\alpha^\mu\cap X_\alpha^v$ and  $\delta\in X_\beta^\mu\cap X_\beta^w$. Since $X_\alpha\times I= \bigcup_{\rho\in  \mathrm{Vt}(X_\alpha)} X_\alpha^\rho$ is simplicial by Lemma \ref{SimplexPiling}, we have 
$$
\Sp(\delta;X_\alpha^v)=\Sp(\delta;X_\alpha^\mu)=\Sp(\delta;L^\mu)=\Sp(\delta;X_\beta^\mu)=\Sp(\delta;X_\beta^w). 
$$
\end{proof}

Next we  use translations in the last coordinate to extend the decomposition of Lemma \ref{ComplexPiling} from $X\times I$ to $X\times [M_1,M_2]$. Given  an ordered $p$-complex $X =\bigcup_{\alpha\in \mathfrak{I}} X_{\alpha,\prec_\alpha}\subset V$, define  the ordered $(p+1)$-simplex $X^\gamma=X_{\prec_\gamma}^\gamma$ by
\begin{equation}\label{YUpgamma}
X^\gamma:=X^v_\alpha+(0_V\times \ell) \qquad\big(\gamma=(\alpha,v,\ell), \ \alpha\in \mathfrak{I},\ v\in\mathrm{Vt}(X_\alpha),\ \ell\in\Z \big),
\end{equation}
where $\prec_\gamma$ is defined by \eqref{neworder}. Thus, for $w,w^\prime\in\mathrm{Vt}(X^\gamma)$,
\begin{align}\label{neworderX}
&w\prec_\gamma w^\prime\Longleftrightarrow\\
&\Big[ [\pi_V(w)\prec_\alpha \pi_V(w^\prime)]\ \text{or}\  [\pi_V(w)= \pi_V(w^\prime)\ \text{and}\ \pi_\R(w^\prime)<\pi_\R(w)] \Big].\nonumber
 \end{align}
Note that when    $\gamma\in J$ in   \eqref{InductiveJ}, then $M_1\le\ell<M_2$. 

\begin{lemm}\label{RPiling} Let $M_1< M_2$ be integers, suppose   $X =\bigcup_{\alpha\in \mathfrak{I}} X_{\alpha,\prec_\alpha}  $ is   an ordered simplicial $p$-complex,  and let   $X^\gamma=X_{\prec_\gamma}^\gamma $ be    as in \eqref{YUpgamma} and \eqref{neworderX}.   Then
\begin{equation}\label{SalphaR}
X\times  [M_1,M_2]=\bigcup_{\gamma\in J} X^\gamma=\bigcup_{\gamma\in J} X_{\prec_\gamma}^\gamma 
\end{equation}
 is  an ordered simplicial $(p+1)$-complex.
\end{lemm}
\begin{proof} As in the proof of  Lemma  \ref{SimplexPiling}, applying an affine isomorphism it is easy to see that
$$
 X_\alpha\times [M_1,M_2]=\bigcup_{\substack{M_1\le\ell<M_1\\ v\in \mathrm{Vt}(X_\alpha)}}\big(X^v_\alpha+(0_V\times \ell)\big)
$$ 
 is   an ordered simplicial $(p+1)$-complex. The rest follows the proof of Lemma \ref{ComplexPiling}.
\end{proof}
\noindent We can take $M_1\to-\infty$ and $M_2\to+\infty$ in \eqref{SalphaR}, as  we record next.
\begin{coro}\label{Minfinite} Let $X$ be as in Lemma \ref{RPiling} and let
\begin{equation*}
J_\infty:=\big\{(\alpha,v,\ell)\big|\,\alpha\in  \mathfrak{I},\ v\in \mathrm{Vt}(X_\alpha),\ \ell\in\Z  \big\}.
 \end{equation*}
Then
$X\times  \R=\bigcup_{\gamma\in J_\infty} X^\gamma $
 is  an ordered simplicial $(p+1)$-complex.
\end{coro}

  Lemma \ref{RPiling} shows that any $\rho\in X\times [M_1,M_2]$ belongs to some simplex $ X^\gamma$ with $\gamma=(\alpha,v,\ell)$. Our next result shows that $v$ can be chosen to be a spanning vertex of the projection $\pi_V(\rho)\in X$.
\begin{lemm}\label{Spanv} Let $X,
  M_1$ and $ M_2$ be as in Lemma \ref{RPiling}. Suppose $\rho\in X\times [M_1,M_2]$ and $\pi_V(\rho)\in X_\alpha$. Then there is some $v\in\Sp\big(\pi_V(\rho)\big)\subset\mathrm{Vt}(X_\alpha)$ and an integer $\ell$ satisfying $M_1\le \ell< M_2$ such that $\rho\in X^\gamma$,  where $\gamma=(\alpha,v,\ell)$.
\end{lemm}
\begin{proof} Lemma \ref{RPiling} shows  that $\rho\in X^\Delta$ for some $\Delta=(\alpha,w,\ell)\in  J$.   If $w\times \ell$ or $w\times(\ell+1)$ is a spanning vertex of $\rho$ (with respect to $X^\Delta$), then $w=\pi_V(w\times \ell)=\pi_V\big(w\times(\ell+1)\big)$ is a spanning vertex of $\pi_V(\rho)$, and we may pick $v:=w$. Otherwise,
\begin{equation}\label{spanrho}
\Sp(\rho)=\big\{v_1\times (\ell+1),\ldots,v_t\times(\ell+1),w_1\times \ell,\ldots,w_r\times \ell  \big\},
\end{equation}
where, by      definition \eqref{SalphaR} of $ X^\gamma$ and  definition \eqref{StimesI} of $X_\alpha^v$, the  $v_i $ and $w_j$ are vertices of $X_\alpha$ satisfying
$$
v_1\prec_\alpha v_2\prec_\alpha \cdots\prec_\alpha v_t\prec_\alpha w\prec_\alpha w_1\prec_\alpha w_2\prec_\alpha \cdots\prec_\alpha w_r.
$$
Applying $\pi_V$ to \eqref{spanrho} we obtain
$$
\Sp\big(\pi_V(\rho)\big)=\big\{v_1,\ldots,v_t,w_1,\ldots,w_r\big\}.
$$
If $t>0$, let $v:=v_t$.
Otherwise,   let $v:=w_1$. Then,  from \eqref{spanrho}, \eqref{SalphaR}  and   \eqref{StimesI},
$\Sp(\rho) \subset \mathrm{Vt}( X^\gamma)$. Hence $\rho\in X^\gamma$ and $v\in \Sp\big(\pi_V(\rho)\big)$.
\end{proof}

\subsection{The complex $Y$} Throughout this subsection we assume,   as in Proposition  \ref{MainInduction}, that  $M_1<M_2$ are integers and  that
$X =\bigcup_{\alpha\in \mathfrak{I}} X_{\alpha,\prec_\alpha}  \subset V$ is a simplicial $p$-complex contained in some vector space $V$, with compatible orders $\prec_\alpha$ in the sense of \eqref{Ordered}. So far in this section we have made no use of the linear function $\omega:V\to\R$  in Proposition \ref{MainInduction}. Now we will  use  $\omega$ to make two changes in the  construction of the simplices $X^\gamma$ in Lemma \ref{RPiling}. First we will replace the given  order  $\prec_\alpha$ on the vertices of $X_\alpha$ by a new order $\prec_\alpha^A$ which   depends on  $\omega$. Then, using the piecewise affine map $T$ defined below, we will  make a change in the last coordinate of elements of $X\times\R$   to obtain the simplices $Y_\gamma$ in Proposition \ref{MainInduction}.

Recall that  for $u,u^\prime\in  \mathrm{Vt}(X_\alpha)$ we defined in \eqref{SalphaA0}
  \begin{equation}\label{SalphaA}
u\prec_\alpha^A u^\prime\ \Longleftrightarrow\
\Big[
\big[A(u)<A(u^\prime)\big]\ \ \mathrm{or}\ \ \big[A(u)=A(u^\prime)\ \mathrm{and}\  u \prec_\alpha u^\prime\big] \Big],
  \end{equation}
where $A$ is the upper fractional part   of $\omega$. A trivial verification shows that
if the orders $\prec_\alpha$ are compatible, then so are the orders $\prec_\alpha^A $.

 Applying Lemma \ref{RPiling}  to the ordered  simplicial complex $X:=\bigcup_{\alpha\in\mathfrak{I}} X_{\alpha,\prec_\alpha^A}$,   we obtain the simplicial complex
\begin{equation}\label{Halfnewcomplex}
X\times[M_1,M_2]=  \bigcup_{\gamma\in J} X^{\gamma,A} =:X^A,
  \end{equation}
 where we have  called the   new simplices $X^{\gamma,A}$ (instead of $X^\gamma$) to clarify that  the original order $\prec_\alpha$ on the vertices of $X_\alpha$  has been replaced by $\prec_\alpha^A$.\footnote{\ We do not consider $X^A$ as an ordered complex since this is not  the complex $Y$ appearing in Proposition \ref{MainInduction}. To get $Y$ we will still need to apply the map $T$ studied in Lemma \ref{NiceT}  below.} We have also denoted by $X^A$ the corresponding simplicial decomposition of $X\times [M_1,M_2]$. By Corollary \ref{Minfinite}, we also  get a simplicial complex
\begin{equation*}
X\times\R=  \bigcup_{\gamma\in J_\infty} X^{\gamma,A}=:X^A_\infty.
  \end{equation*}

For $\gamma=(\alpha,v,\ell)\in J_\infty$, \ie $\alpha\in \mathfrak{I}$,  $v\in\mathrm{Vt}(X_\alpha)$ and $\ell\in\Z$, let
$$
T_\gamma:=V\times\R\to V\times\R
$$
 be the unique affine function
which on any  vertex $\sigma \in \mathrm{Vt}\big(X^{\gamma,A} \big)$ satisfies 
\begin{equation}\label{Talphav}
T_\gamma(\sigma) := \sigma\ +\  \big(0_V\times (A(\pi_V(\sigma) )-\omega (\pi_V(\sigma))-1)\big)
 \quad\big(\sigma \in \mathrm{Vt} (X^{\gamma,A} ) \big),
\end{equation}
where $\pi_V:V\times\R\to V$ is the projection to the first component and $0_V$ is the origin in $V$.
 Actually,
$T_\gamma(\sigma)  \in X\times\Z$
since $\sigma\in X\times \Z$ by  \eqref{StimesI} and \eqref{YUpgamma}, while  $A(v)-\omega(v)\in\Z$ by \eqref{Avchar0}.
Since $T_\gamma(\sigma) $ depends only on $\sigma$ (and not on the simplex
$X^{\gamma,A}$   to which it belongs) and $X^A_\infty$ is a simplicial complex,
there is a unique piecewise affine function 
\begin{equation}\label{Talpharho}
 T:X\times \R\to X\times \R,\quad T(\rho)=T_\gamma(\rho)  :=  \sum_{\sigma\in \mathrm{Vt}(X^{\gamma,A})} c_\sigma T_\gamma(\sigma) \quad(\rho\in X^{\gamma,A}),
\end{equation}
where
\begin{equation}\label{Rho}
\rho=\sum_{\sigma\in \mathrm{Vt}(X^{\gamma,A})} c_\sigma \sigma\qquad \qquad \qquad
 \Big(\sum_{\sigma\in \mathrm{Vt}(X^{\gamma,A})} c_\sigma=1,\ \,c_\sigma\ge0\Big).
\end{equation}
Note that $T(\rho)=\rho$ if $\omega$ vanishes identically.
\begin{lemm}\label{NiceT} Suppose $X =\bigcup_{\alpha\in \mathfrak{I}} X_{\alpha,\prec_\alpha}  \subset V$ is a simplicial complex   in some vector space $V$, with compatible orders $\prec_\alpha$.  Then the piecewise affine function $T:X\times \R\to X\times \R$ defined in \eqref{Talpharho} is a bijection. It satisfies the identities
\begin{equation}\label{TVt}
\pi_V\circ T=\pi_V,\quad T\big((s\times t)+(0_V\times t^\prime)\big) =T(s\times t)\,+\,(0_V\times t^\prime)\quad(s\in X,\ t,t^\prime\in\R),
\end{equation}
 and
\begin{equation}\label{Tclear}
T(s\times  t)= s\times (r_s+t) \qquad\qquad \big(s\in X,\ t\in\R,\  s\times r_s :=T(s\times 0)\big).
\end{equation}
\end{lemm}
\begin{proof}
The last identity  implies the first two, and hence that $T$ is a bijection. To prove \eqref{Tclear}, let $\rho=s\times t\in  X^{\gamma,A} $, where $\gamma=(\alpha,v,\ell)$. From \eqref{Talphav},  \eqref{Talpharho} and \eqref{Rho}, we get
\begin{align}\label{Toget}
T(\rho) & = \sum_{\sigma \in \mathrm{Vt}(X^{\gamma,A})}
 c_\sigma \Big( \sigma+\big(0_V\times\big(A(\pi_V(\sigma))-\omega(\pi_V(\sigma))-1)\big)\Big)\nonumber\\ & =
 \rho -\big(0_V\times\omega(\pi_V(\rho))\big)+\big(0_V\times q\big),
\end{align}
where $ q=q(s\times t):=-1+\sum_{\sigma}  c_\sigma  A\big(\pi_V(\sigma)\big) .$

Now,   $\pi_V$ restricted to $\mathrm{Vt}(X^{\gamma,A}) $ is a bijection onto  $\mathrm{Vt}(X_\alpha)$,
   except for the two vertices $v\times \ell$ and $v\times ( \ell+1)$, both of which map to $v$ (see \eqref{YUpgamma} and $(a)$ after the proof of  Lemma \ref{SimplexPiling}). From \eqref{Rho} we have
   $$
   s= \pi_V(s\times t) =
   \sum_{\sigma\in \mathrm{Vt}(X^{\gamma,A})}c_\sigma\pi_V(\sigma)= \sum_{\delta\in\mathrm{Vt}(X_\alpha)} d_\delta\delta ,
   $$
 where in the last equation we have simply written the barycentric expansion of $s$  with respect to $X_\alpha$ $\big($see the remarks after \eqref{Simplex}$\big)$. Since barycentric
coordinates with respect to a simplex are unique, we have $c_{v\times \ell}+c_{v\times ( \ell+1)}=d_v$ and $c_\sigma=d_{\pi_V(\sigma)}$ for $\pi_V(\sigma)\not=v$.
Hence
$$
q(s\times t):=-1+\sum_{\sigma } c_\sigma  A\big(\pi_V(\sigma)\big) =-1+\sum_{\delta\in\mathrm{Vt}(X_\alpha)} d_\delta A(\delta).
$$
But the $d_\delta$ are uniquely determined by $s$, so  $q(s\times t)=q(s\times 0)$. Now \eqref{Tclear}   follows from \eqref{Toget}.\end{proof}

By the lemma just proved, the affine function $T_\gamma:V\times\R\to V\times\R$ is injective when restricted to a non-empty open subset of $V\times\R$, namely on the interior
of any  $X^{\gamma,A}$. Hence $T_\gamma$ is an affine  bijection. Thus
$$
T(X^{\gamma,A})=T_\gamma(X^{\gamma,A})\subset   X \times\R
$$
is a  $(p+1)$-simplex and

$$
T(X\times[M_1,M_2])=\bigcup_{\gamma\in J} T( X^{\gamma,A})
$$
is a   $(p+1)$-complex, the  isomorphic image by $T$ of the $(p+1)$-complex   in  \eqref{Halfnewcomplex}
\begin{equation}\label{XA}
X^A=X\times[M_1,M_2]=\bigcup_{\gamma\in J}  X^{\gamma,A}.
\end{equation}
Since, as  remarked  before  \eqref{Halfnewcomplex}, $X^A$ is a simplicial complex,  so is $T(X^A)$.

We can now prove part of Proposition  \ref{MainInduction}.

\begin{lemm}\label{ProofPartProp}With  the hypotheses and notation of Proposition \ref{MainInduction}, the following hold.

\vskip.1cm
\noindent$\mathrm{(a)}\ $ $Y_\gamma =T(X^{\gamma,A}), \ \,\mathrm{Vt}( Y_\gamma)= T\big( \mathrm{Vt}(X^{\gamma,A} ) \big)\ \,(\gamma\in J)$, and so   by \eqref{XA}, $Y=T(X\times [M_1,M_2])$.

\vskip.1cm
\noindent$\mathrm{(b)}\   Y= Y(X,\omega,[M_1,M_2])=\bigcup_{\gamma\in J} Y_{\gamma,\prec_\gamma} $ is an ordered  simplicial $(p+1)$-complex and the  orders $\prec_\gamma$ are compatible.

\vskip.1cm
\noindent$\mathrm{(c)}\  $Define  $\Omega:V\times\R\to\R$ by $ \Omega(s\times t):=\omega(s)+t$. Then for all  $\kappa\in Y_\gamma$,
\begin{equation*}
A(v)+\ell-1\le\Omega(\kappa)\le A(v)+\ell\qquad\qquad\qquad (\gamma=(\alpha,v,\ell)\in J ).
\end{equation*}
\end{lemm}
 \begin{proof}
The second and third equations in (a)  follow directly from the first one and Lemma \ref{NiceT}. To prove the first equation in (a), we start with the case  $\gamma=(\alpha,v,0)$ with $\alpha\in \mathfrak{I}$ and $v\in\mathrm{Vt}(X_\alpha)$. Unraveling    definition \eqref{Halfnewcomplex},  $X^{\gamma,A}=S^v$, where we have applied Lemma \ref{SimplexPiling} to $S_\prec:= X_{\alpha, \prec_\alpha^A}$. If we order   the vertices of $X_\alpha$ as $v_0\prec_\alpha^A v_1\prec_\alpha^A \cdots\prec_\alpha^A v_p$,  then  by \eqref{StimesI}  the vertices of $X^{\gamma,A}$ are
$$
\big\{v_0\times 1,v_1\times 1,\ldots,v_j\times 1,v_j\times 0,v_{j+1}\times 0,\ldots,v_p\times 0\}\qquad\qquad\qquad(v_j=v).
$$
Applying definition \eqref{Talphav} of $T_\gamma$, the vertices of  $T(X^{\gamma,A})=T_\gamma(X^{\gamma,A})$  are
\begin{align*}
 \{v_0&\times \big(A( v_0)-\omega (v_0)\big),\ldots,v_j\times \big(A( v_j)-\omega (v_j)\big),v_j\times  \big(A( v_j)-\omega (v_j)-1\big),\\ &v_{j+1}\times \big(A( v_{j+1})-\omega (v_{j+1})-1\big),\ldots,
v_p\times \big(A( v_p)-\omega (v_p)-1\big)\big\} .
\end{align*}
Comparing this with the definition \eqref{Newvertices} of the vertices of $Y_\gamma$, we find that we  have proved (a) for
$\gamma=(\alpha,v,0)$. This implies (a) for any $\gamma=(\alpha,v, \ell)$  since  $Y_{(\alpha,v,\ell)}=Y_{(\alpha,v,0)}+(0_V\times\ell)$  by \eqref{Newvertices}, while  by \eqref{YUpgamma} and \eqref{TVt}
$$
X^{(\alpha,v,\ell),A}=X^{(\alpha,v,0),A}+(0_V\times\ell),\quad T\big((s\times t)+(0_V\times\ell)\big)=
T\big((s\times t)\big)+(0_V\times\ell).
$$

As we remarked just before stating Lemma  \ref{ProofPartProp}, the $(p+1)$-complex $\bigcup_{\gamma\in J} T( X^{\gamma,A})$ is simplicial. By (a),  $T([X\times[M_1,M_2])=\bigcup_{\gamma\in J}Y_\gamma=:Y$ is simplicial.
This proves (b) since the assumed compatibility of the $\prec_\alpha$ immediately implies that of  the   $\prec_\gamma$ defined in \eqref{gammaOrder}.

We now prove (c). Just as in the proof of (a), it suffices to take $\gamma=(\alpha,v,0)$.
 Since $\Omega$ is linear, it suffices to prove (c) for all vertices $\kappa\in \mathrm{Vt}(Y_\gamma)$.
 Since  $\ell=0$,  by (a) we have  $\kappa=T(\delta\times t)$, where $\delta\in  \mathrm{Vt}(X_\alpha)$, $t=1$ if $\delta\prec_{\alpha}^A v$ and $t=0$ if $v\prec_{\alpha}^A \delta$. By \eqref{Talphav},
\begin{equation}\label{OmegaVt}
\Omega(\kappa)=\Omega( \delta\times t)+\Omega\big(0_V\times (A(\delta)-\omega(\delta)-1)\big)=t+A(\delta)-1.
\end{equation} There are two vertices with $\delta=v$, namely $T(v\times 0) $ and $ T(v\times1).$
From \eqref{OmegaVt} we have  $\Omega\big(T(v\times0)\big)= A(v)-1$ and  $\Omega\big(T(v\times1)\big)= A(v)$, proving  (c)  for the   vertices $ T(v\times0)$ and $ T(v\times1)$.

We may therefore restrict to $\kappa=T(\delta\times t)\in \mathrm{Vt}(Y_\gamma)$ with $\delta\not=v$.
We consider first the case $A(\delta)=A(v)$. Then  \eqref{OmegaVt} gives
 $\Omega(\kappa)=A(v)-1$ if $t=0$, while if $t=1$ we have  $\Omega(\kappa)=A(v)$. Hence
  (c)  holds if  $A(\delta)=A(v)$. If  $A(\delta)\not=A(v)$, then by definition \eqref{SalphaA} of the order
  $\prec_{\alpha}^A $, we have $t=1$ if and only if $A(\delta)<A(v)$.
If $t=1$, then  (c)  follows from
 $$
 A(v)-1\le0< A(\delta)= \Omega(\kappa)<A(v),
 $$
 where we again used  \eqref{OmegaVt} for the equality.
If $t=0$, so $A(\delta)>A(v)$, then
 $$  A(v)>0\ge A(\delta)-1=\Omega(\kappa)>A(v)-1,$$
 which proves  (c)  in the last remaining case.
\end{proof}
\subsection{$\Lambda$-invariance}\label{Lambdainvariance}
Although we have assumed in Proposition \ref{MainInduction} that the orders $\prec_\alpha$ on Vt$(X_\alpha)$ are invariant with respect to a lattice $\Lambda $, the new orders $\prec_{\alpha}^A$ on  Vt$(X_\alpha)$ in general are  not  $\Lambda$-invariant. We shall prove now that this deviation  is determined by
\begin{equation}\label{deltavl}
\Delta_{v,\lambda}:=  A(v)+A(\lambda)-A(v+\lambda)=
\begin{cases}
0\ &\mathrm{if}\ A(v)+A(\lambda)\le1,\\
1\ &\mathrm{if}\ A(v)+A(\lambda)>1,
\end{cases}
  \end{equation}
where the last equality  used \eqref{Avchar0} and the  linearity of    $\omega$.

\begin{lemm}\label{Aordering}  Let $\Lambda\subset V$ be a  lattice for which   $X =\bigcup_{\alpha\in J} X_{\alpha,\prec_\alpha} $ has  $\Lambda$-invariant orders $\prec_\alpha$ $\big($see \eqref{translation}$\big)$,  let  $\omega:V\to\R$ be a linear function, define  $\prec_{\alpha}^A$  and $\Delta$ as  in \eqref{SalphaA}  and  \eqref{deltavl}. Suppose  $v$ and $w$ are vertices of  $X_\alpha $ and, for some $\lambda\in\Lambda$, $v+\lambda$ and $w +\lambda$ are vertices of   $X_{\alpha^\prime}$  ($\alpha,\,\alpha^\prime\in J$).

Then,  translation  preserves order if $\Delta_{v,\lambda}= \Delta_{w,\lambda}$, \ie
\begin{equation}\label{BAT1}
 v\prec_{\alpha}^A w \ \Longleftrightarrow \   v+\lambda\prec_{\alpha^\prime}^A w+\lambda \qquad \qquad (\Delta_{v,\lambda}= \Delta_{w,\lambda}),
\end{equation}
while it reverses order if $\Delta_{v,\lambda}\not= \Delta_{w,\lambda}$, \ie
\begin{equation}\label{BAT2}
 v\prec_{\alpha}^A w   \ \Longleftrightarrow \   w+\lambda\prec_{\alpha^\prime}^A v +\lambda\ \qquad \qquad ( \Delta_{v,\lambda}\not= \Delta_{w,\lambda}).
\end{equation}
Moreover,
\begin{equation}\label{AveqAw}
 A(v)=A(w)\  \Longleftrightarrow \   A(v+\lambda)=A(w+\lambda) ,
\end{equation}
and
\begin{equation}\label{AveqAw2}
 A(v)=A(w) \ \Longrightarrow \ \Delta_{v,\lambda}= \Delta_{w,\lambda} .
\end{equation}
\end{lemm}
\begin{proof}
From \eqref{deltavl} we obtain
\begin{equation}\label{BAT}
\big( A(v+\lambda)-A(w+\lambda) \big) - \big( A(v)-A(w)\big) = \Delta_{w,\lambda}- \Delta_{v,\lambda}=-1,0,\ \mathrm{or}\ 1.
\end{equation}
As $A$ takes values in $(0,1]$,   \eqref{AveqAw} and \eqref{AveqAw2} follow.

We now turn to the first two claims in the lemma. If $A(v)=A(w)$, then  $\Delta_{w,\lambda}= \Delta_{v,\lambda}$  and   the orders $\prec_{\alpha}^A$ and $\prec_{\alpha^\prime}^A$  in   \eqref{BAT1} reduce  to $\prec_\alpha$  and $\prec_{\alpha^\prime}$, respectively. In this case  \eqref{BAT1} follows from the assumed $\Lambda$-invariance of the order  \eqref{translation}.
We may therefore suppose $A(v)\not=A(w)$. Then $A(v+\lambda)\not=A(w+\lambda)$ and the orders $\prec_{\alpha}^A$ and $\prec_{\alpha^\prime}^A$    in \eqref{BAT1} and \eqref{BAT2}  reduce  to    the usual order of $A$-values $\big($see \eqref{SalphaA}$\big)$. If $\Delta_{w,\lambda}= \Delta_{v,\lambda}$, then \eqref{BAT} shows that
$A(v)-A(w)=A(v+\lambda)-A(w+\lambda)$, so \eqref{BAT1} follows. If $\Delta_{w,\lambda}\not= \Delta_{v,\lambda}$, then $$\Big|\big( A(v+\lambda)-A(w+\lambda) \big)-\big(A(v)-A(w)\big)\Big|= 1.$$ But this is the distance between two  real numbers in the open interval $(-1,1)$, which must therefore   have opposite signs. This proves \eqref{BAT2}.
\end{proof}

 \begin{lemm}\label{Tranlate}
Let $\Lambda\subset V$ be  a full lattice in a vector space $V$, let  $X =\bigcup_{\alpha\in \mathfrak{I}} X_\alpha  \subset V $ be a  $\Lambda$-complex (see Definition $\mathrm{\ref{LambdaComplex}}$),   $\omega:V\to\R$    a linear function, $Y=\bigcup_{\gamma\in J} Y_{\gamma,\prec_\gamma}\subset V\times\R$  as in \eqref{InductiveY} and  $\pi_V:V\times\R\to V$   the  projection. Assume
  that for some $ \alpha\in \mathfrak{I},\  v\in   \mathrm{Vt}(X_\alpha)  $ and $\lambda\in\Lambda$
 we have
   $      \big(v+\lambda\big)\in \mathrm{Vt}(X_{\alpha^\prime}) $ for some $\alpha^\prime\in \mathfrak{I}.$  Suppose also that $\kappa \in  \mathrm{Vt}( Y_{(\alpha,v,\ell)} )$ and $ \big(\pi_V(\kappa)+\lambda\big)  \in \mathrm{Vt}(X_{\alpha^\prime})$.  Then $   (\kappa+\widehat{\lambda})\in \mathrm{Vt}(Y_{(\alpha^\prime,v+\lambda, \ell)}) $, where
\begin{equation}\label{shift}
 \widehat{\lambda} =\lambda\times  {j}\in\Lambda\times\Z,  \quad   {j} := A(\lambda)-\omega(\lambda)-\Delta(v,\lambda),
\end{equation}
 $A$ being as in \eqref{Avchar0} and    $\Delta$ as in \eqref{deltavl}.
\end{lemm}
\noindent Note that $\widehat{\lambda} $ and $\gamma^\prime:=(\alpha^\prime,v+\lambda,\ell)$  depend on $  \lambda$ and on $\gamma:=(\alpha,v ,\ell)$, but not on $\kappa \in  \mathrm{Vt}( Y_\gamma )$.
\begin{proof}  That $ {j}  \in  \Z $ is immediate from \eqref{Avchar0} and  \eqref{deltavl}.
 Now,  $\kappa\in \mathrm{Vt}( Y_\gamma )=T\big( \mathrm{Vt}(  X^{\gamma,A})\big)$, by Lemma  \ref{ProofPartProp} (a). Thus,
$$\kappa=T(w\times t),\quad w\times t\in  \mathrm{Vt} (X^{\gamma,A}),\quad  t\in\{\ell, \ell+1\}  ,\quad w=\pi_V(\kappa)\in \mathrm{Vt}(X_\alpha).
$$
By \eqref{Talphav},
\begin{equation}\label{Tgamma}
\kappa= T(w\times t)= w\times t \ +\ 0_V\times\big(A(w)-\omega(w)-1\big) .
\end{equation}
 By assumption, both  $v+\lambda$ and $w+\lambda$ are vertices of $X_{\alpha^\prime}$. Thus there exists a vertex 
\begin{equation}\label{RHO}
\rho\in  \text{Vt}(  X^{\gamma^\prime,A}),\ \rho =(w+\lambda)\times t^\prime ,\ t^\prime\in\{\ell+1,\ell\},\  \gamma^\prime:=(\alpha^\prime,v+\lambda,\ell) .
 \end{equation}
 Notice that for $w\not=v$, the value of $t^\prime$ in \eqref{RHO} is determined  by $w+\lambda$, but for $w=v$  there are  vertices $\rho =(w+\lambda)\times t^\prime\in  \text{Vt}(  X^{\gamma^\prime,A})$ for both values of $t^\prime$.
 We now calculate
 \begin{align}
&T(\rho) =\rho\  +\     0_V\times\big(A(w+\lambda)-\omega(w+\lambda)-1\big)\qquad   \ \ [\text{by}\ \eqref{Talphav}\text{ and }\eqref{RHO}] \nonumber
\\ &
= (w+\lambda)\times t^\prime  \, +\,     0_V\times\big(A(w)+A(\lambda)-\Delta(w,\lambda)-\omega(w)-\omega(\lambda)-1\big) \nonumber \\
&  \qquad\qquad\qquad\qquad\qquad\qquad\qquad\qquad\qquad\qquad\qquad\qquad\qquad  [\text{by}\, \eqref{deltavl}]  \nonumber \\ &
 = w\times t\  +\     0_V\times\big(A(w)-\omega(w)-1\big)\ +\ \lambda\times\big(A(\lambda) -\omega(\lambda)-\Delta(v,\lambda)\big) \nonumber
 \\ &
 \qquad\quad+\
0_V\times\big(\Delta(v,\lambda)-\Delta(w,\lambda) +t^\prime-t\big)\qquad  \qquad  \qquad   [t \text{ as in  } \eqref{Tgamma}]\nonumber
\\ &
= \kappa \  +\  \widehat{\lambda}  \ \,+\  \,  0_V\times\big(\Delta(v,\lambda)-\Delta(w,\lambda) +t^\prime-t\big)\quad [\text{by}\ \eqref{shift}\text{ and }\eqref{Tgamma}].\nonumber
\end{align}
Since  $ T(\rho)\in  \mathrm{Vt}(Y_{\gamma^\prime}),$   Lemma \ref{Tranlate} will be proved once we show
\begin{equation}\label{needed}
\Delta(v,\lambda)-\Delta(w,\lambda) +t^\prime-t=0.
\end{equation}
We will do this by   checking   various cases, bearing in mind that $\Delta=0$ or $1$   and $t,t^\prime=\ell$ or $\ell+1$.

If $w=v$, both $(w+\lambda)\times \ell$ and $(w+\lambda)\times (\ell+1)$ are
  vertices of $X^{\gamma^\prime,A}$. Hence we can pick the vertex $\rho$ so that $t^\prime=t$, making \eqref{needed} hold  trivially.
  We may therefore assume $w\not=v$, which makes $t$  and $t^\prime$ uniquely determined by $\prec_\alpha^A$ and $\prec_{\alpha^\prime}^A$. Namely,  $t=\ell+1$  is equivalent to $w\prec_\alpha^A v$  (see Lemma \ref{SimplexPiling}). Otherwise, $t=\ell$. Similarly,
  $t^\prime=\ell+1$ if and only if $w+\lambda\prec_{\alpha^\prime}^A v+\lambda$.

  If $\Delta(v,\lambda)=\Delta(w,\lambda)$, Lemma \ref{Aordering} shows $t=t^\prime$, proving \eqref{needed} in this case.   We may therefore assume $\Delta(v,\lambda)\not=\Delta(w,\lambda)$. Thus, again  by Lemma \ref{Aordering}, $t\not=t^\prime$, $A(v)\not=A(w)$ and $A(v+\lambda)
 \not=A(w+\lambda)$.

  We consider first the case $t=\ell$. Then $v\prec_\alpha^A w$, so $A(w)>A(v)$ by definition \eqref{SalphaA} of $\prec_\alpha^A$. Also, $t^\prime\not=t$, so $t^\prime=\ell+1$. We claim $\Delta(w,\lambda)=1$.
  Indeed,  otherwise $\Delta(w,\lambda)=0$ and so $\Delta(v,\lambda)=1$.  Therefore,  by \eqref{deltavl},  $A(v)+A(\lambda)>1$, whence $A(w)+A(\lambda)>1$. Then, again by \eqref{deltavl}, $\Delta(w,\lambda)=1$, a contradiction. Hence, $\Delta(w,\lambda)=1$  and    $\Delta(v,\lambda)=0$, from which  \eqref{needed} follows.

  Lastly, if $t=\ell+1$, then $A(w)<A(v)$ and $t^\prime=\ell$. We claim $\Delta(v,\lambda)=1$. Otherwise  $\Delta(v,\lambda)=0$ and $\Delta(w,\lambda)=1$.  By \eqref{deltavl}, $  A(v)+A(\lambda)\le 1 $. Hence
   $ A(w)+A(\lambda) < 1 $, proving  $\Delta(w,\lambda)=0,$ a contradiction. Thus, $\Delta(v,\lambda)=1, \ \Delta(w,\lambda)=0$,    and  \eqref{needed} follows.   \end{proof}
\subsection{Proof of Proposition \ref{MainInduction}}\label{proofofprop}
We need to show that $Y=\bigcup_{\gamma\in J} Y_{\gamma,\prec_\gamma}\subset V\times \R$ is an ordered $(p+1)$-complex   satisfying inequality \eqref{Tineq}, and  properties $(i)$ to $(v)$ in Definition \ref{LambdaComplex}  with respect to the lattice
$$
\widehat{\Lambda}:=\Lambda\times(M_2-M_1)\Z.
$$
Parts  (b) and (c) of Lemma \ref{ProofPartProp} show  that $Y$ is an ordered simplicial $(p+1)$-complex with compatible orders and satisfies  inequality \eqref{Tineq}. We must still prove that $Y$ has properties $(iii)$ to $(v)$.

We now prove $(iii)$, \ie if $ \kappa, \delta\in \mathrm{Vt}(Y_\gamma)$ and $ \kappa+\widehat{\lambda}, \delta+\widehat{\lambda}\in \mathrm{Vt}(Y_{\gamma^\prime})$ for some $ \widehat{\lambda}\in \widehat{\Lambda},\ \gamma=(\alpha,v,\ell),\ \gamma^\prime=(\alpha^\prime,v^\prime,\ell^\prime)\in J   ,$  then we must show  
\begin{equation}\label{Trans}
 \kappa  \prec_\gamma  \delta \qquad\Longleftrightarrow\qquad \kappa+\widehat{\lambda}  \prec_{\gamma^\prime}  \delta+\widehat{\lambda}  .
\end{equation}
By  \eqref{TVt} $\pi_V\circ T=\pi_V$, so $\pi_V$ maps $\mathrm{Vt}(Y_\gamma)=T\big(\mathrm{Vt}(X^{\gamma,A})   \big)$  to $ \mathrm{Vt}(X_\alpha)$ $\big($see (a) in  Lemma \ref{ProofPartProp}  and remark $(a)$ following the proof of Lemma
\ref{SimplexPiling}$\big)$. Of course,  $\pi_V$ maps $\widehat{\Lambda} $ to $\Lambda$. Since $X$ is  by assumption a  $\Lambda$-complex,
\begin{equation}\label{PiVorder}
\pi_V(\kappa) \prec_\alpha \pi_V(\delta)\qquad\Longleftrightarrow\qquad\pi_V(\kappa+\widehat{\lambda}) \prec_{\alpha^\prime} \pi_V(\delta+\widehat{\lambda}).
\end{equation}
If $\pi_V(\kappa) \not= \pi_V(\delta)$,   then  $\pi_V(\kappa+\widehat{\lambda}) \not= \pi_V(\delta+\widehat{\lambda})$, so \eqref{Trans}  is clear from  \eqref{gammaOrder} and  \eqref{PiVorder}.
When $\pi_V(\kappa) = \pi_V(\delta)$,     \eqref{gammaOrder} defines
$$ \kappa  \prec_\gamma  \delta  \Longleftrightarrow t^\prime<t\qquad\qquad(   \kappa= u\times t , \  \delta= u\times t^\prime).
$$
Hence \eqref{Trans} is again clear and we have proved   $(iii)$.

We now turn to the proof of $(iv)$  in Definition \ref{LambdaComplex}, \ie that $Y$ is nearly a fundamental domain for  $\widehat{\Lambda}\subset V\times\R$.
 By assumption, $X$ surjects onto $V/\Lambda$.  Recall from \eqref{Tclear} that  $T(x\times s)= x\times (r_x+s)$, where $r_x\in\R$ is independent of $s\in\R$. Thus,  $Y=T(X\times [M_1,M_2])$  surjects onto $(V\times\R)/\widehat{\Lambda}$.

To complete the proof of $(iv)$, we show  that  the union of interiors $\bigcup_{\gamma\in J }\stackrel{\circ}{Y_\gamma}$ injects into $(V\times\R)/\widehat{\Lambda}$.  Suppose $\rho=T(s\times t)\in\, \stackrel{\circ}{ Y_\gamma}$ and $ \rho^\prime=T(s^\prime\times t^\prime)\in \,\stackrel{\circ}{Y_{\gamma^\prime}}$  map to the same point in $(V\times\R)/\widehat{\Lambda}$, for some $\gamma=(\alpha,v,\ell),\ \gamma^\prime=(\alpha^\prime,v^\prime,\ell^\prime)\in J   $. Since $\rho$ is in the interior of $Y_\gamma=T(X^{\gamma,A})$ and $\pi_V(Y_\gamma)=X_\alpha$,
 it follows that  $  s=\pi_V(\rho)\in\,\stackrel{\circ}{X_\alpha }$. Similarly,  $s^\prime \in\,\stackrel{\circ}{X_{\alpha^\prime} }$.
But we are assuming the congruence
$$
(s,r_s+t)=T(s\times t)\equiv T(s^\prime\times t^\prime)= s^\prime\times (r_{s^\prime}+t^\prime) \ \ (\mathrm{modulo}\  \Lambda\times(M_2-M_1)\Z),
$$
so $s\equiv  s^\prime  \ (\mathrm{modulo}\  \Lambda).$
 By assumption, $X=\bigcup_\alpha X_\alpha$  injects into $V/\Lambda$. Thus, $s=s^\prime$. Then
$$
 s\times (r_s+t)=T(s\times t)\equiv T(s\times t^\prime)= s\times (r_s+t^\prime) \ \ (\mathrm{modulo}\ \ \widehat{\Lambda}),
$$
showing $t-t^\prime\in(M_2-M_1)\Z$. It $t\not=t^\prime$, then $ t =M_1$ or $M_2$ as  $t,t^\prime\in [M_1,M_2]$. However, $t\in\Z$ means that $s\times t$ lies in a proper face of some $X^{\gamma^{\prime\prime},A}$ (see remark $(b)$ after Lemma \ref{SimplexPiling}). Hence $\rho=T(s\times t)$  lies in a proper face of  $Y_{\gamma^{\prime\prime}}$. Since $Y$ is a simplicial complex, this contradicts $\rho \in\, \stackrel{\circ}{ Y_\gamma}$, proving $(iv)$.

We now turn to the proof of $(v)$, \ie the $\widehat{\Lambda}$-equivariance of the spanning vertices.
To this end, suppose
\begin{equation}\label{StartEscher}
y,y^\prime\in Y=T(X\times [M_1,M_2]) ,\quad y^\prime=y+\widetilde{\lambda},\quad\widetilde{\lambda}=\lambda\times j\in\widehat{\Lambda}.
\end{equation}
We must show that their spanning vertices satisfy
\begin{equation}\label{EscherXI}
\Sp(y^\prime ;Y ) =\widetilde{\lambda}+\Sp(y;Y ),
\end{equation}
where for clarity we have written $\Sp (y ;Y)$ to indicate that the spanning vertices are calculated with respect to $Y$, \ie with respect to any simplex $Y_\gamma\subset Y$ containing $y$.
We claim that it suffices to prove
\begin{equation}\label{Vertices}
\kappa\in \Sp(y ;Y)   \quad\Longrightarrow\quad\kappa+\widetilde{\lambda}\in \text{Vt}(Y_{\widetilde\gamma })
\end{equation}
for some $\widetilde\gamma \in J$  independent of $\kappa$. Indeed,
  the barycentric expansion
 \begin{equation}\label{Ybary}
 y=\sum_\kappa c_\kappa \kappa\qquad\qquad\qquad\Big(\kappa\in\Sp(y;Y) ,\ \ \sum_\kappa c_\kappa=1,\ \ c_\kappa>0\Big)
  \end{equation}
implies $ y^\prime= y +\widetilde{\lambda}=\sum_\kappa c_\kappa (\kappa+\widetilde{\lambda})$, proving   \eqref{EscherXI} since $ \kappa+\widetilde{\lambda}\in \text{Vt}(Y_{\widetilde\gamma})$ by \eqref{Vertices}.

We now   prove  \eqref{Vertices}.
By \eqref{TVt} and \eqref{StartEscher},
$$
 x^\prime =x+\lambda\in X,\quad x:=\pi_V(y)\in X,\quad x^\prime:=\pi_V(y^\prime)\in X,\quad \lambda:=\pi_V(\widetilde{\lambda})\in\Lambda.
$$
Since $X=\bigcup_{\beta\in J} X_\beta$ is by assumption a $\Lambda$-complex,
\begin{equation}\label{EscherX}
\Sp (x^\prime;X )=\lambda+\Sp (x ;X).
\end{equation}
As $y\in Y= T(X\times [M_1,M_2])$, we have $y=T(\rho)\in Y_\gamma$ for some $\gamma=(\alpha,v,\ell)\in J$ and $\rho\in X^{\gamma,A}$. Therefore
$\Sp(y )\subset \text{Vt}( Y_\gamma).$
Similarly, $y^\prime\in Y_{\gamma^\prime} $ for some   $\gamma^\prime=(\alpha^\prime,v^\prime,\ell^\prime)\in J$. By Lemma \ref{Spanv} and \eqref{TVt}, we can assume
$v\in\Sp (x )\subset \mathrm{Vt}(X_\alpha).$

Let $\kappa\in \Sp(y;Y)\subset \mathrm{Vt}(Y_\gamma) $. Then $ \pi_V(\kappa)\in \Sp(x )\subset \mathrm{Vt}(X_\alpha)$. By \eqref{EscherX}, $v+\lambda\in \mathrm{Vt}(X_{\alpha^\prime})$ and
$\pi_V(\kappa)+\lambda\in \mathrm{Vt}(X_{\alpha^\prime})$. By Lemma \ref{Tranlate}, there is a $\widehat{\lambda}=\lambda\times\widehat{j}\in\Lambda\times\Z$, independent of $\kappa$, such that
$\kappa+\widehat{\lambda}\in\mathrm{Vt}(Y_{(\alpha^\prime,v+\lambda,\ell)}).$
By adding  $0_V\times q$  to $\widehat{\lambda}$  for some $q\in\Z$, we may ensure $\widehat{\lambda}\in\widehat{\Lambda}$ and
\begin{equation}\label{kappalambda}
\kappa+\widehat{\lambda}\in\mathrm{Vt}(Y_{(\alpha^\prime,v+\lambda,\ell^\prime)}),
\end{equation}
where $\ell^\prime:=\ell+q$. By again adding  $0_V\times q^\prime(M_2-M_1)$   to $\widehat{\lambda}$ for some $q^\prime\in\Z$, we may assume $\ell^\prime\in[M_1,M_2)$. Note that $\gamma^\prime=(\alpha^\prime,v+\lambda,\ell^\prime)$ is independent of $\kappa$.
Using the barycentric expansion \eqref{Ybary} of $y$, let
\begin{equation}\label{widehaty}
 \widehat{y}:=\sum_{\kappa} c_\kappa(\kappa+\widehat{\lambda})=y+\widehat{\lambda}\in Y_{\gamma^\prime}\subset Y=T(X\times [M_1,M_2]).
\end{equation}
As $y^\prime,\widehat{y}\in Y$, we can write
$$
y^\prime =T(x^\prime\times t^\prime),\qquad \widehat{y}= T(\widehat{x}\times \widehat{t}\,)\qquad(x^\prime,\widehat{x}\in X,\ \, t^\prime,\widehat{t}\in [M_1,M_2]).
$$
Applying the projection $\pi_V$ and \eqref{TVt}, we find
$$
x^\prime=\pi_V(y^\prime)=\pi_V(y+\widetilde{\lambda})=x+\lambda=\pi_V(y+\widehat{\lambda})=\pi_V( \widehat{y})= \widehat{x}.
$$
Therefore, using \eqref{Tclear},
$$
\widetilde{\lambda}-\widehat{\lambda}=y^\prime-\widehat{y}=T(x^\prime\times t^\prime)-T(x^\prime\times \widehat{t}\,)=
 x^\prime\times (r_{x^\prime}+t^\prime)- x^\prime\times (r_{x^\prime}+\widehat{t})= 0_V\times (t^\prime-\widehat{t}).
$$
Since $\widetilde{\lambda}-\widehat{\lambda}\in\widehat{\Lambda}:=\Lambda\times(M_2-M_1)\Z$ and $t^\prime,\widehat{t}\in [M_1,M_2]$, there are only three possibilities:

\begin{enumerate}[i)]
\item $t^\prime=\widehat{t}$,

\item $ t^\prime=M_1,\ \widehat{t}=M_2$,

\item $ t^\prime=M_2,\ \widehat{t}=M_1$.
\end{enumerate}

\noindent Note that $t^\prime$ and $\widehat{t}$ are independent of $ \kappa\in \Sp(y)$ since $\widetilde{\lambda}$ and  $\widehat{\lambda}$ are independent of $\kappa$. Hence the same case   i),  ii) or iii) above  occurs,  independently of $\kappa$.

In   case  i),   $\widetilde{\lambda}=\widehat{\lambda}$, so  \eqref{kappalambda} implies \eqref{Vertices} with $\widetilde\gamma:= \gamma^\prime $. In this case we are done. Consider now  case   ii), so $\widehat{y}=T(\widehat{x}\times M_2)$ and
\begin{equation}\label{bothlambda}
\widetilde{\lambda}=\widehat{\lambda}+\big(0_V\times (M_1-M_2)\big).
\end{equation}
 Order the vertices $v_j$ of $X_{\alpha^\prime} $   so that
 $
 v_0\prec_{\alpha^\prime}^A v_1\prec_{\alpha^\prime}^A\cdots\prec_{\alpha^\prime}^A v_p.
 $
Using the notation of \eqref{RoofFloor},  let
$$
\big(X^{(\alpha,v,\ell),A}\big)^\pm:=\big(X_{\alpha,\prec_\alpha^A}^v\big)^\pm+(0_V\times\ell)\subset X^{(\alpha,v,\ell),A}\qquad\qquad(\ell\in\Z).
$$
From \eqref{TopBottom} we have
$$
 (\widehat{x}\times M_2 )\in  \big(X^{(\alpha^\prime,v_p,M_2-1),A}\big)^+  \subset X^{(\alpha^\prime,v_p,M_2-1),A}.
 $$

 Hence
$$
 \Sp(\widehat{x}\times M_2;X^A)\subset \big(X^{(\alpha^\prime,v_p,M_2-1),A}\big)^+,
$$
where $X^A$ is the simplicial complex $X^A:=\bigcup_\gamma X^{\gamma,A}=X\times [M_1,M_2]$ in \eqref{Halfnewcomplex}. Thus,
\begin{align*}
\widehat{y}=T(\widehat{x}\times M_2)\in T\big((X^{(\alpha^\prime,v_p,M_2-1),A})^+\big)&\subset T(X^{(\alpha^\prime,v_p,M_2-1),A})\\
&=  Y_{(\alpha^\prime,v_p,M_2-1)}.
\end{align*}
As
 $\widehat{y}:=\sum_\kappa c_\kappa(\kappa+\widehat{\lambda})$ by \eqref{widehaty}, we see from \eqref{kappalambda}     that
\begin{equation}\label{kappawidehatlambda}
\kappa+\widehat{\lambda}\in \Sp(\widehat{y} )\subset \mathrm{Vt}\big( T((X^{(\alpha^\prime,v_p,M_2-1),A})^+)\big).
\end{equation}
By \eqref{TopBottom} again,
\begin{equation*}
(X^{(\alpha^\prime, v_p,M_2-1),A})^+=(X^{(\alpha^\prime, v_0,M_1),A})^- + \big(0_V\times(M_2-M_1)\big).
\end{equation*}
Applying $T$ and \eqref{TVt} we find
\begin{equation}\label{ALmost}
T\big((X^{(\alpha^\prime,v_p,M_2-1),A})^+\big)=T\big((X^{(\alpha^\prime,v_0,M_1),A})^- \big)+ \big(0_V\times(M_2-M_1) \big).
\end{equation}
Using \eqref{bothlambda}, \eqref{kappawidehatlambda}    and \eqref{ALmost} we obtain
\begin{equation*}
  \kappa+\widetilde{\lambda}=\kappa+\widehat{\lambda}+\big(0_V\times(M_1-M_2)\big)  \subset
 \mathrm{Vt}\big((Y_{(\alpha^\prime,v_0,M_1)})^-\big)\subset  \mathrm{Vt}\big(Y_{(\alpha^\prime,v_0,M_1)} \big),
\end{equation*}
proving  \eqref{Vertices} with $\widetilde{\gamma}:=(\alpha^\prime,v_0,M_1)$.

Case   iii)  is analogous, for then $\widehat{y}=T(\widehat{x}\times  M_1 )\in (Y_{(\alpha^\prime,v_0,M_1)})^-$ and $\widetilde{\lambda}=\widehat{\lambda}+\big(0_V\times(M_2-M_1)\big)$. Reasoning as above, we find
$\kappa+\widetilde{\lambda}\in \mathrm{Vt}\big(Y_{(\alpha^\prime,v_p,M_2-1)}\big)$, concluding the proof  of Proposition \ref{MainInduction}.

\section{Degree theory and the signed fundamental domain}\label{topology}
\noindent
In this section we prove Theorem \ref{Maintheorem}. Thus  we have   fixed independent totally positive units $\varepsilon_\ell\in k$,  principal logarithms $  \log\varepsilon_\ell^{(r_1+j)} $ of these units at the complex embeddings, and integers $N_j\ge3$ $\,(1\le\ell\le r:=r_1+r_2-1, \ \,1\le j\le r_2)$.
   Lastly, we have also fixed a twister function $\beta:\Z^{n-1}\to k^*$, as defined in \S\ref{Twistersdefined}.
\subsection{The    function $f$ on  $\mathfrak{X}\times\R$}\label{SPN}
\begin{prop}\label{BigX} The $(n-1)$-complex $\mathfrak{X}=\bigcup_{\alpha\in \mathfrak{I}}\mathfrak{X}_\alpha\subset\R^{n-1}$   defined in  \S$\mathrm{\ref{ComplexX}}$ is a $ \Lambda$-complex  (see Definition $\mathrm{\ref{LambdaComplex}}$), where $ \Lambda :=\Z^{r}\times N_1\Z\times\N_2\Z\times\cdots\times N_{r_2}\Z \subset\Z^{n-1}$. It has  the  following  additional   properties.

\noindent$(i)$ The vertices have integral coordinates, \ie  $\mathrm{Vt}(\mathfrak{X}_\alpha)\subset\Z^{n-1}$ for all $\alpha\in\mathfrak{I}$.

\noindent$(ii)$ There are $\displaystyle | \mathfrak{I}|=(n-1)!\prod_{j=1}^{r_2} N_j$  simplices $\mathfrak{X}_\alpha$.


\noindent$(iii)$   For  any integer   $1\le j\le r_2$ and any index  $\alpha\in \mathfrak{I}$, there is an $\alpha_j\in\R$ such that any
 $\kappa \in\mathfrak{X}_\alpha$ satisfies $\alpha_j\le \Omega_j(\kappa)\le \alpha_j+1$, where
\begin{equation}\label{OMEGAJ}
 \Omega_j(\kappa):=\kappa[r+j]+\frac{N_j}{2\pi}\sum_{\ell=1}^r \kappa[\ell]\arg\!\big(\varepsilon_\ell^{(r_1+j)}\big) \quad(  \kappa=(\kappa[1],...,\kappa[n-1])\in  \R^{n-1} )
\end{equation}
and $-\pi< \arg\!\big(\varepsilon_\ell^{(r_1+j)}\big)\le\pi$.
\end{prop}
\begin{proof}
Recall from \S\ref{ComplexX} that  $\mathfrak{X}:=X_{n-1}$ was constructed inductively starting from the trivial one-point complex $X_0$,   letting 
\begin{align*}
X_j:=Y(X_{j-1},\omega_{j-1},[0,Q_j])  
\end{align*}
where $Q_j :=1, \  \omega_{j-1}(x):=0$ if $1\le j\le r$,   and
$$
Q_{r+j} :=N_j,\qquad \omega_{r+j-1}(x):=
\frac{N_j}{2\pi} \sum_{\ell=1}^r x[\ell]\arg\!\big(\varepsilon_\ell^{(r_1+j)}\big)\quad \qquad(1\le j <r_2).
$$
Since $X_0$ is obviously a $\Lambda_0$-complex for the trivial lattice $\Lambda_0:=\{0\}$, Proposition \ref{MainInduction}
and induction on $j$ shows that $X_j\subset\R^j$ is a $Q_1\Z\times\cdots\times Q_j\Z$-complex for $1\le j\le n-1$.\footnote{\ We identify $0\times\R^j$ with $\R^j$, of course.} Taking $j=n-1$ shows that $\mathfrak{X}$ is a $ \Lambda $-complex.

Property $(i)$ in Proposition \ref{BigX} was already proved in \eqref{neededsimplices}, while $(ii)$ follows from \eqref{InductiveCard}. To prove $(iii)$, let  $\kappa\in\mathfrak{X}_\alpha$  and let
$\pi_{r+j}:\R^{n-1}\to\R^{r+j}$ be the projection to the first $r+j$ coordinates ($1\le j\le r_2$).
Then $\pi_{r+j}(\kappa)$ is contained in some $ (r+j)$-simplex $Y_\gamma:=\pi_{r+j}(\mathfrak{X}_\alpha)\subset X_{r+j}$ for some index $\gamma=(\beta,v,\ell)$   depending only on $j$ and $\alpha$. Here the complex $X_{r+j}$ was defined in \eqref{rjcomplexdefined}, $v$ is a vertex in a simplex in $X_{j+r-1}$ indexed by $\beta$, and $\ell$ is an integer in the range $0\le\ell<N_j$. Inequality \eqref{Tineq} in Proposition  \ref{MainInduction},  applied to   $\pi_{r+j}(\kappa)\in Y_\gamma,$ to  $X:=X_{r+j-1},$ $V=\R^{r+j-1},$ $\omega:=\omega_{r+j-1}$ and to $[M_1,M_2]:=[0,N_j]$, yields $(iii)$ with
$\alpha_j:=A(v)+\ell-1 $ and $A(v)$ as in \eqref{Tineq}.
\end{proof}

We now define the function $f$   in    \eqref{doblediagram}. Henceforth we keep the notation of Proposition \ref{BigX}, \ie $ \mathfrak{X},  \mathfrak{I}, \Lambda$ and $\Omega_j$. For $\alpha\in \mathfrak{I}$, let $A_\alpha:\R^{n-1} \to \R^{r_1}\times\C^{r_2}$ be the unique
 affine function which on vertices $ v=(v[1],...,v[n-1] ) \in \mathrm{Vt}(\mathfrak{X}_\alpha) \subset\Z^{n-1}$ is given by
 \begin{equation}\label{tildef}
A_\alpha(v):=\tilde{f}(v):=\beta(v)\prod_{\ell=1}^r \varepsilon_\ell^{v[\ell]}\in k^*\subset  \R^{r_1}\times\C^{r_2}.
\end{equation}
Here we used $(i)$ of Proposition \ref{BigX} and the twister function $\beta$ recalled at the beginning of this section.
Note that since  by definition $\beta$ is $\Lambda$-periodic,
\begin{equation}\label{tildefpluslambda}
\tilde{f}(v+\lambda)=\tilde{f}(v)\prod_{\ell=1}^r\varepsilon_\ell^{\lambda[\ell]}\qquad\qquad(v\in\Z^{n-1},\ \lambda\in\Lambda).
\end{equation}
 Let $A:\mathfrak{X}\to \R^{r_1}\times\C^{r_2}$ be the   piecewise affine function which equals $A_\alpha$ on $\mathfrak{X}_\alpha$. Since $ \mathfrak{X}=\bigcup_{\alpha\in \mathfrak{I}} \mathfrak{X}_\alpha $ is a simplicial complex by Proposition \ref{BigX}, and  $\tilde f(v)$ in \eqref{tildef} does not  depend on $\alpha\in  \mathfrak{I}$, the $A_\alpha$ define a unique (and continuous)  piecewise affine function $A:\mathfrak{X}\to\R$ which restricts to $A_\alpha$ on $ \mathfrak{X}_\alpha $. Let
   \begin{equation}\label{f}
 f(x\times y):= \e^y A(x), \qquad\quad\qquad f:\mathfrak{X}\times\R\to\R^{r_1}\times\C^{r_2} ,
\end{equation}
where $ \e^y$ is regarded as a real scalar multiplying each one of the $r_1+r_2$ coordinates of $A(x)$.
Writing   $ x=\sum_{i=0}^{n-1} c_i v_i$, where $ \sum_{i=0}^{n-1} c_i=1, \  c_i\ge0 \  \,(0\le i\le n-1)$, \eqref{affinemap} gives
 \begin{equation}\label{affinef}
f(x\times y)=\e^y\sum_{i=1}^n c_i w_i\ \big(x\in\mathfrak{X}_\alpha=[ v_0,...,v_{n-1}],\  y\in\R,\ w_i:=\tilde f(v_i) \big).
\end{equation}

We now prove some basic properties of $f$.
\begin{lemm}\label{fprop}  The function $f$ defined in \eqref{f} above has the following properties.
\begin{equation}\label{frange}
(i)\qquad\qquad f( \mathfrak{X}_\alpha\times\R)\subset  \mathcal{H}_\alpha:=\R_+^{r_1}\times H\Big(\frac{\alpha_1+\frac12}{N_1}\Big)
 \times\cdots\times H\Big(\frac{\alpha_{r_2}+\frac12}{N_{r_2}}\Big),
 \end{equation}
 where $ \alpha_j\in\R$ was defined in Proposition $\mathrm{\ref{BigX}}\  (iii)$   and where $H(t)\subset\C^*$ is the open half-plane
\begin{equation*}
H(t)  :=\big\{r \e^{2\pi i (t+ \theta)} \in\C^*\big|\,r>0,\  \theta \in  (-1/4 ,1/4 )\big\}.
\end{equation*}
Thus,
\begin{equation} \label{frange1}
f( \mathfrak{X} \times\R)\subset  \R_+^{r_1}\times {\C^*}^{r_2} .
\end{equation}
 
\noindent$(ii)$  For $\alpha\in J$,   let $v_0,...,v_{n-1}$ be the vertices of $\mathfrak{X}_\alpha$ (in any order) and  let $w_i:= \tilde f(v_i)= f(v_i\times0) \in k\subset\R^{r_1}\times\C^{r_2}$, and let $\overline{C}_\alpha$ be as defined in \eqref{closedcone}. Then
\begin{align}\label{fcone}
&\overline{C}_\alpha=f(\mathfrak{X}_\alpha\times\R)\\
&\nonumber=\Big\{t_0w_0+t_1w_1+\cdots+t_{n-1}w_{n-1}\big|\, t_i\ge0  \text{ for }0\le i\le n  -1 \Big\}\ - \big\{0\big\},
\end{align}
 the closed $k$-rational polyhedral cone generated by the $w_i$, minus the origin.

\noindent$(iii)$   If $y\in\R$ and $x^\prime=x+ \lambda$ with $x,x^\prime\in \mathfrak{X}$ and $ \lambda\in  \Lambda$  $\big($see Proposition $\ref{BigX}\big)$, then  \begin{equation}\label{fescher}
f(x^\prime\times y)=\varepsilon( \lambda )  f(x\times y), \  \varepsilon( \lambda ) :=\prod_{\ell=1}^r\varepsilon_\ell^{\lambda[\ell]}\ \big(\lambda=(\lambda[1],...,\lambda[n-1])  \big).
\end{equation}
\end{lemm}
\noindent Note that   $\varepsilon( \lambda ) $ is a unit  belonging to  the subgroup generated by   $\varepsilon_1,...,\varepsilon_r $. The  product $\varepsilon( \lambda )  f(x\times y)$  is taken in the ring $\R^{r_1}\times{\C}^{r_2}.$ Recall that we regard  $k\subset\R^{r_1}\times{\C}^{r_2}$.
\begin{proof}
 We begin with the proof of
    \eqref{frange}, \ie   $f(\mathfrak{X}_\alpha\times\R)\subset\mathcal{H}_\alpha$. The set  $\mathcal{H}_\alpha$ is a  convex cone,  so  from \eqref{affinef} it is clear that  it suffices to prove   $  f(\kappa\times0)\in \mathcal{H}_\alpha$ for $\kappa\in\text{Vt}(\mathfrak{X}_\alpha)$. From \eqref{tildefpluslambda} and \eqref{affinef},
\begin{equation}\label{fkappa}
 f(\kappa\times0)^{(j)}=\beta(\kappa)^{(j)}\prod_{\ell=1}^{r}\big( \varepsilon_\ell^{(j)}\big)^{\kappa[\ell]}\qquad\qquad(1\le j\le r_1+r_2).
\end{equation}
 For $1\le j\le r_1$, $\beta(\kappa)^{(j)}>0$ and  $\varepsilon_\ell^{(j)}>0$, so
$ f(\kappa\times0)^{(j)}>0$.
For $1\le j\le  r_2$, recall from   \eqref{beta0} in the definition of twisters,
\begin{align}
\label{tj} &\beta(\kappa)^{(r_1+j)}\\
&\nonumber=|\beta(\kappa)^{(r_1+j)}| \exp\!\big(2\pi i(\kappa[r+j]+t_j(\kappa))/N_j\big)\quad\Big(\frac{|t_j(\kappa)|}{N_j}< \frac14-\frac{1}{2N_j }\Big).
\end{align}
Using Proposition \ref{BigX} $(iii)$ and \eqref{OMEGAJ},
we can write  
\begin{equation}\label{thetaj}
 \Omega_j(\kappa):=\kappa[r+j]+\frac{N_j}{2\pi}\sum_{\ell=1}^r \kappa[\ell]\arg\!\big(\varepsilon_\ell^{(r_1+j)}\big)=\alpha_j+\frac12+\theta_j(\kappa)
\end{equation}
for some $\theta_j(\kappa)\in[-1/2,1/2]$. Then,  letting $ r_j(\kappa):=\big|\tilde f(\kappa )^{(r_1+j)}\big|>0$, from \eqref{fkappa} we have
\begin{align*}
&\tilde f(\kappa )^{(r_1+j)}\\
&=r_j(\kappa) \exp\!\big(2\pi i(\kappa[r+j]+t_j(\kappa))/N_j\big)\exp\!\Big(i\sum_{\ell=1}^r  \kappa[\ell]\arg( \varepsilon_\ell^{(r_1+ j)})\Big)
\\
&=r_j(\kappa) \exp\!\bigg(\frac{2\pi i}{N_j}\Big(\kappa[r+j] +\frac{N_j}{2\pi}\sum_{\ell=1}^r  \kappa[\ell]\arg( \varepsilon_\ell^{(r_1+ j)}) \Big) \bigg)\exp(2\pi i t_j(\kappa)/N_j)
\\
&=r_j(\kappa) \exp\!\big( 2\pi i\Omega_j(\kappa)/N_j\big)\exp(2\pi i t_j(\kappa)/N_j)
\\
&=r_j(\kappa) \exp\!\big( 2\pi i( \alpha_j+\textstyle{\frac12)}/N_j\big)\exp\!\big(2\pi i (t_j(\kappa)+\theta_j(\kappa))/N_j
 \big).
\end{align*}
By \eqref{tj} and \eqref{thetaj},
$|t_j(\kappa)+\theta_j(\kappa)|/N_j< \frac14$,
showing that $\tilde f(\kappa )^{(r_1+j)}\in H\big( (\alpha_j+\frac12)/N_j\big)$ and proving \eqref{frange}.

To prove \eqref{fescher}, write the barycentric expansion of $x=\sum_{i}c_i v_i\in\mathfrak{X}_\alpha$ as in \eqref{affinef}, but only include
vertices $v_i$ with $c_i>0$, \ie $v_i\in\Sp(x )$. Since $\mathfrak{X}$ is a $\Lambda$-complex by Proposition \ref{BigX}, every $v_i+\lambda$ is a spanning vertex of $x^\prime$, and therefore a vertex of some $\mathfrak{X}_{\alpha^\prime}$ where $\alpha^\prime$ is  independent of $v_i$
 $\big($see property $(v)$ in Definition \ref{LambdaComplex}$\big)$. Hence we obtain  the barycentric expansion of $x^\prime$ as
$$
x^\prime=x+\lambda= \lambda+\sum_{i}c_i v_i=\sum_{i}c_i (v_i+\lambda).
$$
By   \eqref{affinef} and \eqref{tildefpluslambda},
$$
f(x^\prime\times y)=\e^y\sum_{i}c_i \tilde{f}(v_i+\lambda)=\e^y\sum_{i}c_i \tilde{f}(v_i)\varepsilon(\lambda)
=\varepsilon(\lambda)f(x\times y),
$$
proving  \eqref{fescher}.

We now prove  \eqref{fcone}.   Since $f(x\times0)=A_\alpha(x)$ for  $x\in\mathfrak{X}_\alpha $, and  $A_\alpha$ is an affine function, it follows that  $f(\mathfrak{X}_\alpha\times 0)=[w_0,...,w_{n-1}]$, the   convex hull of the $w_i:=A_\alpha(v_i)$.  From \eqref{f}, $f(x\times y)=\e^y f(x\times0)$. It follows that $f(\mathfrak{X}_\alpha\times \R)\subset \mathcal{H}_\alpha$ is the closed cone generated by the $w_i$, minus the origin. It does not contain the origin since $0\notin\mathcal{H}_\alpha$. \end{proof}

\subsection{Global and local degree computations}
 As in Theorem \ref{Maintheorem}, let  $\varepsilon_1,\ldots, \varepsilon_r $ be fixed independent  totally positive units of $k$, where $r=r_1+r_2-1$. We use them to define   a   homomorphism
  $g   $  from the additive group $\R^n$ to the
multiplicative group $\R_+^{r_1}\times {\C^*}^{r_2} $. Namely, let the  $j$-th coordinate of $g\ \,(1\le j\le r_1+r_2)$ be
  \begin{align}\label{Basicgj}
g(x)^{(j)}:=
 \begin{cases}
 \exp\!\big(x[n]+\sum_{\ell=1}^r x[\ell]\log \varepsilon_\ell^{(j)} \big)   &\text{if }  j\le r_1,\\
\exp\!\big(x[n] + \frac{2\pi i x[j+r_2-1]}{N_{j-r_1}}+\sum_{\ell=1}^r x[\ell]\log \varepsilon_\ell^{(j)}\big)  &\text{if }r_1< j .
\end{cases}
  \end{align}
Here $N_j$ is as in  Theorem \ref{Maintheorem} and $\log z$ denotes the  branch with  $-\pi <\im(\log z)\le \pi $.

 \begin{lemm}\label{Crooked} The  homomorphism $g$ in \eqref{Basicgj} is onto, infinitely differentiable and
\begin{equation}\label{kerg}
g^{-1}\big(\langle\varepsilon_1,...,\varepsilon_r\rangle\big)=\widehat{\Lambda}:=  \Lambda\times0= \Z^r\times N_1\Z\times\cdots \times N_{r_2}\Z\times0\subset\R^n,
\end{equation}
where $\langle\varepsilon_1,...,\varepsilon_r\rangle$ is the subgroup of the units of $k$ generated by the $\varepsilon_\ell$.
\end{lemm}
\begin{proof}
From   \eqref{Basicgj} it is clear that  $g$ is infinitely differentiable. Let us  show that $g$
 is onto. Let $\{E_i\}_{i=1}^n$ be the standard basis for $\R^n$. Then $g(E_i)=\varpsilon_i$ for
   $1\le i\le r$,  and $g(E_n)=(\e,...,\e)$.
 Define   $ \Log:\R_+^{r_1}\times {\C^*}^{r_2}\to\R^{r_1+r_2}  $ at $z=\big( z^{(1)},...,z^{(r_1+r_2)} \big)$  by
\begin{equation}\label{Log}
\big(\Log(z)\big)_j:=e_j\log|z^{(j)} |\qquad (e_j:=1 \ \text{if}\ 1\le j\le r_1,\ e_j:=2 \ \text{if}\ r_1 <j\le r_1+r_2).
\end{equation}
Then $\Log\big(g(x)\big)=x[n] F+\sum_{\ell=1}^r x[\ell] \Log(\varepsilon_\ell)$, where  $F_j:=e_j$.
As the $\varepsilon_i$ are assumed independent and $F$ does not lie in the span of the $L(\varepsilon_\ell)$, given $z \in \R_+^{r_1}\times {\C^*}^{r_2}$, there exists $x\in\R^n$ such that $\Log\big(g(x)\big)=\Log(z)$. But then all coordinates of  $zg(-x)\in \R_+^{r_1}\times {\C^*}^{r_2}$ have  absolute value 1, so $z^{(j)}=  g(x)^{(j)}\exp(2\pi i\theta_j)$  with  $\theta_j\in\R$ ($r_1<j\le r_1+r_2$). On letting $x^\prime[\ell]:=x[\ell]+N_{\ell-r}\theta_{\ell-r_2+1}$ for $r< \ell \le n-1$ and   $x^\prime[\ell]:=x[\ell]$ for all  other  $\ell $, 
we have $g(x^\prime)=z$.

A similar use of $\Log$ shows that $x\in\ker(g)$ if and only if $x[j]=0$ for $1\le j\le r$ and $j=n$, and $x[r+j]\in N_j\Z$ for $1\le j\le r_2$. Hence \eqref{kerg} follows.
\end{proof}

\begin{lemm}\label{gprop}  With $g$ and  $  \mathcal{H}_\alpha$ as in \eqref{Basicgj} and \eqref{frange},
 $g( \mathfrak{X}_\alpha\times\R)\subset  \mathcal{H}_\alpha .$
\end{lemm}
\begin{proof} For $1\le j\le r_1$ it is clear that $g(x\times y)^{(j)}>0$ for $x\in\R^{n-1}$ and $y\in\R$. For $1\le j\le r_2$, observe that $g(x\times y)^{(r_1+j)}=\e^yg(x\times0)^{(r_1+j)}$. Hence it suffices to show that $g(x\times0) \in\mathcal{H}_\alpha  $ for $x\in \mathfrak{X}_\alpha$. Now  $g(x\times0)^{(r_1+j)}=r_{x,j}\exp\!\big(2\pi i\Omega_j(x)/N_j\big)$, where $r_{x,j}>0$ and
$\Omega_j(x)$ was defined in \eqref{OMEGAJ}.  Property $(iii)$ in Proposition \ref{BigX} concludes  the  proof since $N_j\ge 3$ by assumption.
\end{proof}
Let $\widehat{\Lambda}\subset\R^n$ be as in \eqref{kerg}, and define the $n$-manifolds
\begin{equation}\label{manifolds}
\widehat{T}:=\R^n/\widehat{\Lambda},\qquad \ T:=(\R_+^{r_1}\times{\C^*}^{r_2})/\langle\varepsilon_1,...,\varepsilon_r\rangle.
\end{equation}
It follows from Lemma \ref{Crooked} that $g$ induces a  bijection $G:\widehat{T}\to T $,
\begin{equation}\label{diagrams} 
\begin{CD}
\R^n @> g > > \R_+^{r_1}\times{\C^*}^{r_2}\\
@VV\widehat{\pi}V @VV\pi V\\
\widehat{T} @>G > > T
\end{CD} \qquad\qquad\qquad 
\end{equation}
making the  diagram  commute. Here
  $\widehat{\pi}$ and $\pi$  are the natural quotient maps.
\begin{lemm}\label{Proper} The  map $G:\widehat{T}\to T$ defined above is a homeomorphism.
\end{lemm}
\begin{proof}
Since   $\widehat{T}$ is not compact, it is not quite obvious that $G^{-1}$ is continuous. But,
as $G$ is a continuous and bijective map between   locally compact metric spaces, it suffices to prove that $G$ is proper. Thus, let $K\subset T$ be compact, and let us show that  $G^{-1}(K)\subset \widehat{T}$ is compact. Define the ``norm" map $\mathcal{N}:\R_+^{r_1}\times{\C^*}^{r_2}\to\R_+$ by
\begin{equation}\label{NORM}
\mathcal{N}(z):=\prod_{j=1}^{r_1+r_2}{\big|z^{(j)}\big|}^{e_j}\qquad\qquad\qquad\big(  z=(z^{(1)},...,z^{(r_1+r_2)})\big),
 \end{equation}
with $e_j=1$ or 2 as in \eqref{Log}. Then  $\mathcal{N}(\varepsilon_i)=1\ \,(1\le i\le r) $, so $\mathcal{N}$ induces a continuous map $\widetilde{\mathcal{N}}:T\to\R_+$. Since $K$ is compact, $\widetilde{\mathcal{N}}(K)\subset[a,b]$ for some positive real numbers $a$ and $b$.
As  $\mathcal{N}\big(g(x[1],...,x[n])\big)=\e^{n x[n]}$, we have   by Lemma \ref{Crooked} and \eqref{manifolds},
$$G^{-1}(K)\subset \widehat{\pi}\big([0,1]^r\times[0,N_1]\times\cdots\times[0,N_{r_2}]\times[ (\log a)/n, (\log  b)/n]\big),$$
a compact subset of $T$.
\end{proof}

From  Lemma \ref{Proper} and   $\widehat{T}:=\R^n/\widehat{\Lambda}$, it is clear that  $\widehat{T}$ and $T$ are homeomorphic to an  $(n-1)$-torus $\times\,\R$. Thus,  $G$ is a homeomorphism between connected, orientable $n$-manifolds.
We fix any orientation of $\R^n $ and orient  $\R^{r_1}\times \C^{r_2}$  by declaring the  homeomorphism $\mathcal{I}:\R^{r_1}\times \C^{r_2}\to\R^n$,
\begin{align}\label{Iiso}
&\mathcal{I}(z_1,... ,z_{r_1+r_2})\\
&:=
\nonumber\big(z_1,...,z_{r_1},\re(z_{r_1+1}),\im(z_{r_1+1}),...,\re(z_{r_1+r_2}),\im(z_{r_1+r_2})\big),
\end{align}
to be orientation-preserving.
The orientation on $\R^{r_1}\times \C^{r_2}$ induces an orientation on its open subset $\R_+^{r_1}\times{\C^*}^{r_2}=\mathrm{domain}( \pi)$.
Finally, we orient $\widehat{T}$ and $T$ by declaring  $\pi$ and $\widehat{\pi}$ in \eqref{diagrams} to be orientation-preserving.

In practice, this means for the homeomorphism $G$ that $\deg(G)=\pm1$ is the sign of the  determinant of the Jacobian matrix Jac$_x(\mathcal{I}\circ g)$ of
$\mathcal{I}\circ g:\R^n\to\R^n$ at any point $x$ of its domain, a number  which we now compute.\footnote{\ The degree of any (proper, continuous) function from an oriented manifold $L$ to itself is independent of the orientation on  $L$. Thus,  $\deg(\mathcal{I}\circ g )$ is independent of the orientation  of $\R^n$ fixed above.  A summary of degree theory in five pages, based on Dold's textbook \cite{Do} and sufficient for our purposes, can be found in \cite[\S7]{DF2}.}
\begin{lemm}\label{Globaldegree} With the orientation on $\widehat{T}$ and $T$ defined above,
 $$ \deg(G)=(-1)^{n-1}(-1)^{r_2(r_2-1)/2}\,\mathrm{sign}\big(\!\det(R)\big)\not=0,$$
where $R=(r_{i j})$ is the 
matrix defined in \eqref{Reg1}.
\end{lemm}
\begin{proof}
We have  $ \det(R) \not=0$ since  the rows with $i>1$ span the  hyperplane
$$ \big\{x=(x_j)_j\in\R^{r_1+r_2}\big|\,x_1+\cdots +x_{r_1}+2x_{r_1+1}+\cdots+  2 x_{r_1+r_2}=0\big\},$$  while the first row is off the hyperplane.
Since $G$ is a homeomorphism, $\deg(G)=\text{locdeg}_p(\widehat{T})$, the local degree at any $p\in\widehat{T}$
\cite[(39)]{DF2}. As $\widehat{\pi}$ and $\pi$ are orientation-preserving local homeomorphisms,   $\deg(G)=\text{locdeg}_x(g)$ for any $x\in\R^n$. By definition of the orientation on $\R_+^{r_1}\times{\C^*}^{r_2}$, $\text{locdeg}_x(g)=\text{locdeg}_x(\mathcal{I}\circ g)$, with $\mathcal{I}$ as in \eqref{Iiso}.
 If   the differential  $d(\mathcal{I}\circ g)_x:\R^n\to\R^n$ is invertible at some $x$, then   \cite[Prop.\ 22]{DF2} $$\deg(G)=\text{locdeg}_x(\mathcal{I}\circ g)=\text{sign}\big(\!\det( d(\mathcal{I}\circ g)_x)\big)=\pm1.$$

To compute the sign of $\det\!\big(\text{Jac}_x(\mathcal{I}\circ g)\big)$, note that  \eqref{Basicgj} and \eqref{Iiso} show that the  coordinates corresponding to $1\le j\le r_1 $ are
$$
\big((\mathcal{I}\circ g)(x)\big)_j=\e^{x[n]+\sum_{\ell=1}^r x[\ell]\log ( \varepsilon_\ell^{(j)} )} =\e^{x[n]+\sum_{\ell=1}^r x[\ell]\log|\varepsilon_\ell^{(j)}|}
$$
 since  the $\varepsilon_\ell$ are assumed positive at all real embeddings. For the complex embeddings  we have two real coordinates, corresponding to the real and imaginary parts of $g(x)^{(j)}$.Thus, setting $\varTheta_\ell^{(j)}:= \arg\!\big( \varepsilon_\ell^{(j+r_1)}\big)$ for $1\le j\le r_2$,
\begin{align*}
\big((\mathcal{I}\circ g)(x)\big)_{2j-1+r_1}&=\e^{x[n]+\sum_{\ell=1}^r x[\ell]\log|\varepsilon_\ell^{(j)}|}
\cos\!\big(\textstyle{\frac{2\pi x[j+r ]}{N_{j }}}+\sum_{\ell=1}^r x[\ell]\varTheta_\ell^{(j)}  \big),\\
\big((\mathcal{I}\circ g)(x)\big)_{2j+r_1}&=\e^{x[n]+\sum_{\ell=1}^r x[\ell]\log|\varepsilon_\ell^{(j)}|}
\sin\!\big(\textstyle{\frac{2\pi x[j+r ]}{N_{j }}}+\sum_{\ell=1}^r x[\ell]\varTheta_\ell^{(j)}  \big).
\end{align*}
It is now easy to calculate  the matrix $J:=\mathrm{Jac}_{x=0}(\mathcal{I}\circ g) $. Namely, abbreviating $L_\ell^{(j)}:= \log|\varpsilon_\ell^{(j)}|\ \,(1\le j\le r_1+r_2,\ 1\le \ell\le r)$, 
$$
J=
\begin{pmatrix}
L_1^{(1)} &\cdots & L_1^{(r_1)} &L_1^{(r_1+1)} &\varTheta_1^{(1)}&\cdots&L_1^{(r_1+r_2)} &\varTheta_1^{(r_2)}\\
L_2^{(1)} &\cdots&L_2^{(r_1)}&L_2^{(r_1+1)} &\varTheta_2^{(1)}&\cdots&L_2^{(r_1+r_2)} &\varTheta_2^{(r_2)}\\
\vdots&\vdots&\vdots&\vdots&\vdots&\vdots&\vdots &\vdots\\
L_r^{(1)} &\cdots&L_r^{(r_1)}&L_r^{(r_1+1)} &\varTheta_r^{(1)}&\cdots&L_r^{(r_1+r_2)} &\varTheta_r^{(r_2)}\\
0 &\cdots&0&0&2\pi/N_1&\cdots&0&0\\
\vdots&\vdots&\vdots&\vdots&\vdots&\vdots&\vdots &\vdots\\
0 &\cdots&0&0&0&\cdots&0&2\pi/N_{r_2}\\
1&\cdots&1&1&0&\cdots&1 &0
\end {pmatrix}.
$$
To compute $\det(J)$, permute the rows 
so that the bottom row  becomes the first row, and row   $i$  becomes row $i+1\ \,(1\le i\le n-1)$. This introduces a factor of  $(-1)^{n-1}$. Now permute columns   so that all   columns coming from imaginary parts at complex places are placed   to the right of all those coming from the real parts (with the places in the same relative order  as before). This gives a factor of    $(-1)^{\sum_{j=1}^{r_2-1} j}=(-1)^{r_2(r_2-1)/2}$, and results in a lower-right $r_2\times  r_2$ diagonal block and  a lower-left block of 0's of size $r_2\times (r_1+r_2)$. 
Hence,
 $$ \det (J)=(-1)^{n-1}(-1)^{r_2(r_2-1)/2}\det(R)\prod_{j=1}^{r_2}(2\pi/N_j)\not=0,$$ from which the lemma follows.
\end{proof}

We now incorporate the map $f$ of Lemma \ref{fprop} into our maps between manifolds. By \eqref{frange1} and \eqref{fescher},   $f$   induces a (continuous) map $F:\widehat{T}\to T $ making
\begin{equation}\label{diagramf} 
\begin{CD}
\mathfrak{X}\times\R @> >f> \R_+^{r_1}\times{\C^*}^{r_2}\\
@VV\widehat{\pi}V @VV\pi V\\
\widehat{T} @> >F > T
\end{CD}
\qquad\qquad\qquad
\end{equation}
a commutative diagram \eqref{diagramf},	where   $\widehat{T},\, \widehat{\pi},\, T$ and $\pi$ are as in \eqref{diagrams}.
\begin{lemm}\label{Homotopy}  Fix the orientation on $\widehat{T}$ and $T$ as in Lemma \ref{Globaldegree}. Then $F$  in  \eqref{diagramf}  is a proper map and
 $$ \deg(F)=\deg(G)=(-1)^{n-1}(-1)^{r_2(r_2-1)/2}\,\mathrm{sign}\big(\!\det(R)\big)=\pm1.$$
\end{lemm}
\begin{proof} By  Lemma \ref{Globaldegree},  it suffices to prove that $F$ is proper and that $\deg(F)=\deg(G)$. This will follow once we construct a  proper homotopy between $G$ and $F$. Suppose $x\times y\in \mathfrak{X}\times\R$. Then $x\in \mathfrak{X}_\alpha$ for some $\alpha\in \mathfrak{I}$. By Lemmas \ref{fprop} and \ref{gprop},  both $f(x\times y)$ and $g(x\times y)$ lie in the same convex set $\mathcal{H}_\alpha\subset  \R_+^{r_1}\times{\C^*}^{r_2}$. Hence we can define $\vartheta:\mathfrak{X}\times\R\times[0,1]\to \R_+^{r_1}\times{\C^*}^{r_2}$, 
\begin{equation*}
\vartheta(x\times y\times t):=tf(x\times y)+(1-t)g(x\times y)\qquad\qquad(x\in \mathfrak{X}, \ y\in\R,\  t\in[0,1]) .
\end{equation*}
Note that $$\vartheta(x\times y\times t)=\e^y\vartheta(x\times 0\times t),$$
as $f(x\times y)=\e^y f(x\times0)$ and $g(x\times y)=\e^y g(x\times0)$.
 For $x\in \mathfrak{X}, \ y\in\R$ and $  \lambda \in \Lambda:=\Z^r\times\Z^{N_1}\times\cdots\times\Z^{N_{r_2}} $, we have by \eqref{fescher} and \eqref{Basicgj}, 
\begin{equation}\label{eschertheta}
 f\big((x+  \lambda)\times y\big)=\varepsilon( \lambda)  f(x \times y) ,\qquad  g\big((x+  \lambda)\times y\big)=\varepsilon(  \lambda)  g(x \times y)        .
\end{equation}
As $\varepsilon( \lambda)  \in E:=\langle\varepsilon_1,...,\varepsilon_r\rangle$,  the homotopy $\vartheta$ descends to the quotient manifolds,
$$
\varTheta:\widehat{T}   \times [0,1]\to   T,\qquad \varTheta\big(  \widehat{\pi}(x\times y)\times t\big) = \pi\big( \vartheta(x\times y\times t)\big).
$$
Since $\varTheta$ provides a homotopy between $F$ and $G$, our only remaining task is to show that  $\varTheta$ is a proper map \cite[Prop.\ 21 (4)]{DF2}. To this end, let $K\subset T$ be compact. We must prove that $\widehat{K}:= \varTheta^{-1}(K)\subset\widehat{T}\times[0,1] $ is compact.  From this it will follow  that  $F$ is proper, since
$F^{-1}(K)= \{\widehat{x}\in\widehat{T}|\, \widehat{x}\times1\in  \varTheta^{-1}(K)\}$.

As the $\Lambda$-complex $\mathfrak{X}\subset\R^{n-1}$ is nearly a fundamental domain for $\Lambda$ (see $(iv)$ in Definition \ref{LambdaComplex}),
\begin{equation}\label{pihatsurj}
\widehat{\pi}(\mathfrak{X}\times\R )=\widehat{T},
\end{equation}
 so it suffices to prove the compactness of
$$\
\mathcal{K}:=\widehat{\pi}^{-1}(\widehat{K})=\vartheta^{-1}\big(\pi^{-1}(K)\big)\subset \mathfrak{X}\times\R\times[0,1].
$$
As  in the proof  of Lemma \ref{Proper}, $\widetilde{\mathcal{N}}(K)\subset[a,b]$ for some $a,b>0$. Since $L:= \pi\big(\vartheta(\mathfrak{X}\times0\times[0,1])\big)$ is a compact subset of $T$,
 $\,\widetilde{\mathcal{N}}(L)\subset[c,d] $ for some $c,d>0$. If $x\times y\times t\in \mathcal{K}$, then
$$\
[a,b]\ni\widetilde{\mathcal{N}}\big(\pi(\vartheta(x\times y\times t ))\big)=\e^{ny}\widetilde{\mathcal{N}}\big(\pi(\vartheta(x\times0\times t ))\big)\in \e^{ny}[c,d].
$$
It follows that $\frac1n\log(a/d)\le y\le \frac1n\log(b/c)$. Thus, $y$ is bounded and  $ \mathcal{K}$ is compact.
\end{proof}

We now  calculate another Jacobian determinant.
\begin{lemm}\label{Affinedegree}
Given $n$ affinely independent elements $v_0,v_1,...,v_{n-1}\in \R^{n-1}$, and any $n$ elements $w_0,w_1,...,w_{n-1}\in\R^{r_1}\times\C^{r_2}$, let
$B:\R^{n-1}\to \R^{r_1}\times\C^{r_2}$ be the unique affine function such that $B(v_i)=w_i\ \,(0\le i\le n-1)$. Define $K:\R^n\to \R^n$ by
$$
K(x[1],...,x[n]):=\mathcal{I}\big(\e^{x[n]}B (x[1],...,x[n-1] )\big) ,
$$
where the linear isomorphism $\mathcal{I}:\R^{r_1}\times\C^{r_2}\to\R^n$ was defined in \eqref{Iiso}. Then
\begin{equation}\label{detf}
\det\!\big(\mathrm{Jac}_x(K)\big) =(-1)^{n-1}\e^{nx[n]}\det(W)/\det(V) 
\end{equation}
where $x=(x[1],...,x[n])\in\R^n,\ W$ is the $n\times n$ matrix whose $i$-th column
is $\mathcal{I}(w_{i-1})\ \,(1\le i\le n )$, and $V$ is the invertible $(n-1)\times(n-1)$ matrix whose $i$-th column is $v_i -v_0\ \,(1\le i\le n-1)$.
\end{lemm}
\begin{proof}
The matrix $V$ is invertible since,  by definition of   affine independence, the $v_i-v_0\ \,(1\le i\le n-1)$ are linearly independent. By \eqref{affinemap},  $B $ satisfies
$B\big(\sum_{i=1}^{n}y_i v_i \big)= \sum_{i=1}^{n }y_i w_i$ provided $\sum_{i=1}^{n}y_i =1, \ y_i\in\R$.
 Let $  E_n:=(0,0,...,0,1)\in\R^n $ and $\tilde{v}_{i}:=v_i\times0\in\R^n=\R^{n-1}\times\R$, and define  $U:\R^n\to\R^n$  by
$$
U(x[1],...,x[n]) : = x[n] E_n+\tilde{v}_0+ \sum_{i=1}^{n-1} x[i](\tilde{v}_i-\tilde{v}_0).
$$
 Then, for $1\le j\le r_1$,
$$
\big((K\circ  U)(x)\big)[j] =   \e^{x[n]}\Big(w_0^{(j)} + \sum_{i=1}^{n-1}x[i] (w_i -w_0)^{(j)}\Big) ,
$$
while for $1\le j\le r_2$,
\begin{align*}
 &\big((K\circ  U)(x)\big)[2j-1+r_1]  =     \e^{x[n]}\Big( \re\big(w_0^{(r_1+j)}\big) + \sum_{i=1}^{n-1}x[i] \re\big((w_i -w_0)^{(r_1+j)}\big)\Big)   , \\
&\big((K\circ  U)(x)\big)[2j+r_1]  =    \e^{x[n]}\Big(\im\big(w_0^{(r_1+j)}\big) + \sum_{i=1}^{n-1}x[i] \im\big((w_i -w_0)^{(r_1+j)}\big)\Big)  .
\end{align*}
The above formulas and  elementary row operations  yield
\begin{align*}
\det\!\big(\text{Jac}_x(K\circ U)\big) &= \e^{nx[n]}(-1)^{n-1} \det\!\big(\mathcal{I}(w_0),...,\mathcal{I}(w_n)\big)\\
&= \e^{nx[n]}(-1)^{n-1} \det(W) .
\end{align*}
Now,
\begin{align*}
\det\!\big(\text{Jac}_x(K\circ U)\big)&=\det\!\big(\text{Jac}_{U(x)}(K)\big)\det\!\big(\text{Jac}_x(U)\big)\\
 &=\det\!\big(\text{Jac}_{U(x)}(K)\big) \det(V),
\end{align*}
so
$$
\det\!\big(\text{Jac}_{U(x)}(K)\big)=\e^{nx[n]}(-1)^{n-1}  \frac{\det(W)}{\det(V)}= \e^{nU(x)[n]}(-1)^{n-1} \frac{\det(W)}{\det(V)}.
$$
Since $U$ is surjective, the proof is done.
\end{proof}

We now use \eqref{diagramf} to calculate the local degree of $F$. 
\begin{lemm}\label{Localdegree} Let $\widehat{x}=\widehat{\pi}(x)\in \widehat{T}$, where $x=\kappa\times y\in\mathfrak{X}_\alpha\times\R$, let $v_0,...,v_{n-1}$ be the vertices of $\mathfrak{X}_\alpha$ $($in any order$)$ and  let $w_i:=  f(v_i\times0)$. Assume that the $w_i$ are $\R$-linearly independent and that   $\kappa$ is an interior point of $\mathfrak{X}_\alpha$. Then the local degree of $F$ at
$\widehat{x}$ is defined and
\begin{equation}\label{LOCDEG}
\mathrm{locdeg}_{\widehat{x}} (F)=
  (-1)^{n-1}\mathrm{sign}\big(\!\det(V)\big)\,\mathrm{sign}\big(\!\det(W)\big)=:d_\alpha,
\end{equation}
 with $V$ and $W$ as in Lemma $\mathrm{\ref{Affinedegree}}$.
\end{lemm}
\begin{proof}By   Lemma \ref{Affinedegree} and the inverse function theorem, $f$ is a local homeomorphism in a neighborhood of $x$, so   $\mathrm{locdeg}_x (f)=\pm1$ is defined.
Note that $\deg(\mathcal{I})=1$ as  $\mathcal{I}$  in \eqref{Iiso} is an orientation-preserving homeo\-morphism, by definition of the orientation of $ \R_+^{r_1}\times{\C^*}^{r_2}$. Hence \eqref{detf}  shows  that $\mathrm{locdeg}_x (f)$ is given by the right-hand side of \eqref{LOCDEG}. Here we used    the fact that the local degree is invariant under composition with orientation-preserving local homeomorphisms \cite[Prop.\ 21 (7)]{DF2}.

We now consider the  commutative diagram \eqref{diagramf}. Note that  $\widehat{\pi}$  is a local orientation-preserving homeomorphism in a neighborhood of $x$  since $\mathfrak{X}$ is nearly a fundamental domain (see property  (iv)  in Definition \ref{LambdaComplex}). The same holds for  the covering map $\pi$ at any point of its domain, and in particular  at $f(x)$. Thus, $\mathrm{locdeg}_{\widehat{x}} (F)= \mathrm{locdeg}_x( f) $.
\end{proof}

Next we prove  that the number of points in any orbit $E\cdot x$ inside a cone $\overline{C}_\alpha$ is bounded independently of $x$, as claimed at the end of Theorem \ref{Maintheorem}.

\begin{lemm}\label{BoundedCard} Let $\overline{C}_\alpha:=f(\mathfrak{X}_\alpha \times\R) $ be the cone  in \eqref{fcone} and let $E:=\langle\varepsilon_1,...,\varepsilon_r\rangle$ be the subgroup
of the units of $k$ generated by the $\varepsilon_\ell$. Then there exists $c_\alpha\in \N$ such that for any $z\in\R_+^{r_1}\times{\C^*}^{r_2}$,  the orbit $E\cdot z$ has at most $c_\alpha$ elements in $\overline{C}_\alpha$. Moreover, given any compact subset $K\subset \R_+^{r_1}\times{\C^*}^{r_2}$, there are at most finitely many $\varepsilon\in E$ such that $\varepsilon K\cap   \overline{C}_\alpha  \not=\varnothing$.
\end{lemm}
\begin{proof} We begin by proving the first claim. Let $\mathcal{N}$ be the ``norm" map defined in \eqref{NORM}. As $\overline{C}_\alpha$ is a cone, the map $z\to z^\dagger:= \mathcal{N}(z)^{-1/n} z$  gives a bijection
 between $\big(E\cdot z\big)\cap \overline{C}_\alpha$ and
 $\big(E\cdot z^\dagger\big)\cap \overline{C}_\alpha$. Hence we may assume $ \mathcal{N}(z)=1$.

 Let $m$ and $M$ be respectively the minimum and maximum values of $\mathcal{N}$ on  the convex hull $[w_0,...,w_{n-1}]$
of the generators $w_i$ of $\overline{C}_\alpha$. Then $m,M>0$, and by homogeneity we have
$$
m  \Big(\sum_{i=0}^{n-1} t_i\Big)^n\le \mathcal{N}\Big(\sum_{i=0}^{n-1} t_i w_i\Big)\le  M\Big(\sum_{i=0}^{n-1} t_i\Big)^n\qquad\qquad(t_i\ge0, \ 0\le i\le n-1).
$$
Hence $\Gamma:=\ker(\mathcal{N})\cap \overline{C}_\alpha$ is compact. Applying the logarithmic map $L$ in \eqref{Log}, we are reduced to bounding the number of $\varepsilon \in E$ such that $L(z)+L(\varepsilon)\in L(\Gamma)$. This is bounded independently of $z$ since the lattice $L(E)$ is discrete, proving the first claim.

The second claim is proved similarly. Namely, if  $\varepsilon K\cap  \overline{C}_\alpha\not=\varnothing$ with $K$ compact, then
 $\varepsilon K^\dagger\cap  \overline{C}_\alpha\not=\varnothing$. As $\varepsilon K^\dagger\cap  \overline{C}_\alpha\subset \Gamma:=\ker(\mathcal{N})\cap \overline{C}_\alpha$, we need to bound   the number of $\varepsilon \in E$ such that $L(z)+L(\varepsilon)\in L(\Gamma)$ for $z\in K^\dagger$. But this is again bounded since  $K^\dagger$ is   compact.
\end{proof}

We can now prove that generic points  $z\in\R_+^{r_1}\times{\C^*}^{r_2}$ satisfy the basic count formula \eqref{BasicCount} in Theorem \ref{Maintheorem}.
More precisely, let $\partial \mathfrak{X}_\alpha:=\mathfrak{X}_\alpha-\stackrel{\circ}{\mathfrak{X}}_\alpha$ be the boundary of $\mathfrak{X}_\alpha$ and let
\begin{equation}\label{generic}
B:= \bigcup_{\substack{\alpha\in \mathfrak{I}\\ \mu_\alpha=0} } (  \mathfrak{X}_\alpha\times\R )  \ \cup\  \bigcup_{\substack{\alpha\in \mathfrak{I}\\  \mu_\alpha\not=0} }  ( \partial \mathfrak{X}_\alpha\times\R ) ,  \qquad\qquad \mathcal{B}:=\bigcup_{\varepsilon\in E} \varepsilon f(B)  ,
\end{equation}
where $E:=\langle \varepsilon_1,...,\varepsilon_r \rangle $ and $\mu_\alpha$ was defined in \eqref{mualphadefined}. Note that $ \mathcal{B}\subset \R_+^{r_1}\times{\C^*}^{r_2}$  is a subset of Lebesgue measure 0 since the cone $\overline{C}_\alpha$ is degenerate if   $ \mu_\alpha=0$.

We now prove the following claims.
 
\noindent$(a)$  If   $\mu_\alpha\not=0$, then $f$ maps $\mathfrak{X}_\alpha\times\R $ bijectively onto the  cone $\overline{C}_\alpha$ in \eqref{fcone}.

\noindent$(b)$ The restriction of $\widehat{\pi}$ to $\mathfrak{X}\times\R-B$ is a bijection onto  $\widehat{T} - \widehat{\pi}(B)$.

\noindent$(c)$  $F$  is   surjective.
 
\noindent Because of  \eqref{fcone}, to prove $(a)$ it suffices to show that $f$ is injective on $\mathfrak{X}_\alpha\times\R$. Let
$x=\sum_{i=0}^{n-1} c_i v_i \in \mathfrak{X}_\alpha$ with $\sum_i c_i=1$ be the barycentric expansion of $x\in \mathfrak{X}_\alpha$. Similarly, let $x^\prime=\sum_i  c_i^\prime v_i\in \mathfrak{X}_\alpha$ and suppose
$f(x\times y)=f(x^\prime\times y^\prime)$. From \eqref{affinef} we find
$
\e^y\sum_i c_i w_i=\e^{y^\prime}\sum_i c_i^\prime w_i.
$
As we are assuming $\mu_\alpha\not=0$, the $w_i$ are linearly independent. Hence $\e^y\  c_i =\e^{y^\prime}  c_i^\prime \ (0\le i\le n-1).$ Summing over $i$ and using $\sum_i c_i=1=\sum_i c_i^\prime$ gives
$y=y^\prime$ and  $c_i=c_i^\prime$, proving claim  $(a)$.

 Claim $(b)$ follows since   the quotient map  $\R^{n-1}\to\R^{n-1}/\Lambda$ is surjective when restricted to $ \mathfrak{X}$ and  injective when restricted to $ \bigcup_{\alpha\in\mathfrak{I}} \stackrel{\circ}{\mathfrak{X}}_\alpha$ (see Proposition \ref{BigX}  and $(iv)$ in Definition \ref{LambdaComplex}). 
Claim $(c)$ follows from   Lemma \ref{Homotopy}, as    $\deg(F)\not=0$  implies the surjectivity of $F$ \cite[Prop. 21 (3)]{DF2}.

\begin{lemm}\label{GenericBasicCount} Let $z\in  \R_+^{r_1}\times{\C^*}^{r_2}- \mathcal{B}$, where $\mathcal{B}$ was defined in \eqref{generic}. Then
$$
 \sum_{\substack{\alpha\in \mathfrak{I}\\ \mu_\alpha\not=0}} \mu_\alpha\sum_{\substack{\varepsilon\in E\\ \varepsilon \cdot z\in \overline{C}_\alpha}} 1=1,
$$
where $\mu_\alpha=\ d_\alpha/\deg(F)$ was given   in  \eqref{mualphadefined} using \eqref{LOCDEG} and Lemma \ref{Homotopy}.
\end{lemm}
\begin{proof}
With the notation of \eqref{diagramf}, set
 $\gamma:=\pi( z)\in T$ and suppose $\widehat{x}\in F^{-1}(\gamma)\subset\widehat{T}$. Such $\widehat{x}$ exists since $F$ is surjective by $(c)$ above.  By \eqref{pihatsurj},  $\widehat{x}=\widehat{\pi}(x)$ for some $x\in\mathfrak{ X}\times\R$. But
$$
\pi\big(f(x)\big)= F\big(\widehat{\pi}(x)\big)=F(\widehat{x})=\gamma=\pi(z),
$$
shows that $f(x)=\varepsilon z$ for some $\varepsilon \in E$. As we assumed $z\notin  \mathcal{B}$,
 we have $x\notin B$. By $(b)$  above, $\widehat{x}$ uniquely determines $x$. Conversely, given   $x\in\mathfrak{ X}\times\R$ such that
$f(x)=\varepsilon z$ for some $\varepsilon \in E$, then  $\widehat{x}:=\widehat{\pi}(x)\in  F^{-1}(\gamma)$.
As $f(x)\in \overline{C}_\alpha$ for some $\alpha\in \mathfrak{I}$ with $\mu_\alpha\not=0$, Lemma \ref{BoundedCard} and
$(a)$ above show that the set of such $\widehat{x}$ is finite and that $F$ is a local homeomorphism in a neighborhood of $\widehat{x}$ . The local-global principle  \cite[Prop. 21 (9)]{DF2}, Lemma \ref{Localdegree}  and $(a)$ above  give
\begin{align*}
\deg(F)&=\deg_{\gamma}(F)=\sum_{\substack{x\in \mathfrak{X}\times\R\\ \widehat{\pi}(x)\in  F^{-1}(\gamma)}} \text{locdeg}_{\widehat{x}}(F)= \sum_{\substack{ x\in\mathfrak{ X}\times\R \\  f(x)\in E \cdot z}}  \text{locdeg}_{\widehat{\pi}(x)}(F)\\ &
= \sum_{\substack{\alpha\in \mathfrak{I}\\ \mu_\alpha\not=0}} d_\alpha\sum_{\substack{x\in \mathfrak{X}_\alpha\times\R\\ f(x)\in E \cdot z}} 1
=\sum_{\substack{\alpha\in \mathfrak{I}\\ \mu_\alpha\not=0}} d_\alpha\sum_{\substack{\varepsilon\in E\\ \varepsilon \cdot z\in \overline{C}_\alpha}} 1.
\end{align*}
\end{proof}

\subsection{End of proof of Theorem \ref{Maintheorem}}

Having proved the basic count \eqref{BasicCount} for  generic $z$, we   extend it to all $z\in  \R_+^{r_1}\times{\C^*}^{r_2}$ following Colmez's unpublished  idea for selecting   boundary  parts of  the   $\overline{C}_\alpha$'s.  It is only here that we must finally assume that $k$ has at least one real embedding.

\begin{lemm}\label{e1}  Let  $v_1,...,v_\ell\in k\subset V:=\R^{r_1}\times \C^{r_2}$ be elements of $k$ with $\ell<n:=[k:\Q]$,  assume  $r_1\ge1$, and define  $e_1\in V$ by $e_1^{(1)}=1$ and $e_1^{(j)}=0$ if $2\le j\le r_1+r_2$. Then $e_1$ is not contained in the real span of $v_1,...,v_\ell$.
\end{lemm}
\begin{proof}Following the proof in \cite[Lemma 9]{DF2}, define the $\R$-bilinear form on $V$,
\begin{equation}\label{trace}
\langle z,w\rangle:=\sum_{j=1}^{r_1} z^{(j)} w^{(j)}+2\sum_{j=r_1+1}^{r_1+r_2}\re\big( z^{(j)} w^{(j)}\big)\qquad\qquad(z,w\in V).
\end{equation}
Note $\langle z,w\rangle=\mathrm{Trace}_{k/\Q}(zw)$    for $z,w\in k$. As $\ell<n$, there exists $x\in k^*$ such that Trace$(xv_i)=0$ for $1\le i\le \ell$. Let $\psi:V\to \R$ be the $\R$-linear function
$\psi(w):=\langle x,w\rangle$. Thus,  $\psi(v_i)=0$, and so  $\psi(e_1)=0$ if $e_1\in\R v_1+...+\R v_\ell$. But \eqref{trace} and $r_1\not=0$ give $x^{(1)}=\langle x,e_1\rangle=\psi(e_1)=0$,  contradicting $x\in k^*$.
\end{proof}

  For any subset $C\subset V$ and $x,y\in V$, we shall say that $\overrightarrow{x ,y} $
pierces $C$  if    $\,y\in C$
  and the closed line segment  $\overrightarrow{x ,y} $ connecting $x$
  and $y$ intersects the interior of $C$. We  now characterize piercing of a cone in terms of  coordinates \cite[Lemma  14]{DF2}.

\begin{lemm}\label{PierceCone}  Let $w_0,...,w_{n-1}\in V$ be a basis of a real vector space $V$, let
\begin{equation*}
\overline{C}=\overline{C}(w_0,...,w_{n-1}):=\big\{t_0w_0+\cdots+t_{n-1}w_{n-1}\big|\, t_i\ge0  \text{ for }0\le i\le n-1   \big\}
\end{equation*}
be the corresponding closed  polyhedral cone. In the basis $w_0,...,w_{n-1}$, write $x=\sum_{j=0}^{n-1} x_j w_j$ and $y=\sum_{j=0}^{n-1} y_j w_j$ . Then
$$
 \overrightarrow{x ,y} \mathrm{ \ pierces \  }\overline{C}  \quad  \Longleftrightarrow \quad \Big[
 y_j\ge0 \ \, ( 0\le j\le n-1)  \mathrm{  \ and \  }\big[y_j=0\Rightarrow x_j>0\big]\Big].
$$
Furthermore, if $ \overrightarrow{x ,y} $ pierces $\overline{C}$ and $s\in  \overrightarrow{x ,y}$ is an interior point of
 $\overline{C}$, then every point of $ \overrightarrow{s ,y}$ is an interior point of
 $\overline{C}$, except possibly for $y$.
\end{lemm}
\begin{proof}
Suppose $\overrightarrow{x ,y} $ pierces  $\overline{C}$. Then $y_j\ge0$ since $y\in  \overline{C}$. Let $s\in \overrightarrow{x ,y}  $ be in the interior of $ \overline{C}$. Then $s=(1-t)x+ty$ for some $t\in[0,1]$ and    $s=\sum_j s_j w_j$ with
$s_j>0\ (0\le j\le n-1)$. But $s_j=(1-t)x_j+ty_j>0$ implies $x_j>0$ whenever $y_j=0$. Conversely, if $y_j\ge0  $, and $y_j=0\Rightarrow x_j>0$, then  for some sufficiently small positive $\epsilon$ and all $t\in [1-\epsilon,1)$, the point $s:=(1-t)x+ty$  lies in the interior of $ \overline{C}$. To prove the last claim in the lemma, assume $s_j>0\ (1\le j\le n)$.  Then $\big((1-t)s+ty\big)_j\ge (1-t)s_j>0$ if $t\in[0,1)$.
 \end{proof}
Lemma \ref{PierceCone}  shows that an equivalent definition of the cone $C_\alpha$ in \eqref{cone} is
$$
C_\alpha=\big\{ y\in\overline{C}_\alpha\big|\, \overrightarrow{  e_1 ,y} \, \text{ pierces }\, \overline{C}_\alpha \big\}.
$$
This is useful  because of the  following  ``piercing invariance" of $e_1$. 
\begin{lemm}\label{Piercee1}  Assume $r_1>0$, let  $y\in V:= \R^{r_1}\times{\C}^{r_2}$ and let $\varepsilon\in  \R_+^{r_1}\times{\C^*}^{r_2}$. Then
 $\overrightarrow{e_1 ,y} $
pierces a closed polyhedral cone $\overline{C}=\overline{C}(w_0,...,w_{n-1})\subset V$  if and only if   $\overrightarrow{\varepsilon e_1 ,y} $ pierces $\overline{C}$.
\end{lemm}
\begin{proof} We may assume that the $n$ generators of    $\overline{C}$ are $\R$-linearly  independent, for otherwise there is no piercing at all and the lemma is trivial.  As $\varepsilon e_1 =\varepsilon^{(1)} e_1$, where the real scalar $\varepsilon^{(1)} >0$, the lemma follows from Lemma \ref{PierceCone}.
 \end{proof}

To complete the proof of Theorem \ref{Maintheorem}  we will prove the basic count \eqref{BasicCount} in the form
 \begin{equation}\label{BasicCount2}
1 =\sum_{\substack{\alpha\in \mathfrak{I}\\ \mu_\alpha\not=0}} \mu_\alpha  \sum_{\varepsilon\in I_\alpha(y)}  1= \sum_{\substack{\alpha\in \mathfrak{I}\\ \mu_\alpha\not=0}} \mu_\alpha  \text{Card}\big( I_\alpha(y)\big)  \quad\quad(y\in \R_+^{r_1}\times{\C^*}^{r_2}),
\end{equation}
where we have set
 \begin{equation*}
I_\alpha(y):=  \big\{\varepsilon\in E\big|\,  \varepsilon y\in  C_\alpha \big\}=\big\{\varepsilon\in E\big|\,  \overrightarrow{e_1, \varepsilon y}\, \text{ pierces }\, \overline{C}_\alpha      \big\}.
\end{equation*}
Note that Lemma \ref{GenericBasicCount} established \eqref{BasicCount2}    only for $y\in \R_+^{r_1}\times{\C^*}^{r_2}-\mathcal{B}$. If $\mu_\alpha=0$  we have $I_\alpha(y) =\varnothing$, for in this case $\overline{C}_\alpha $ has an empty interior.
As in  \cite[Lemma 25]{DF2},  we will prove that the $I_\alpha$ stabilize along the path from $e_1$ to $y$.
\begin{lemm}\label{Jalpha}  For $t\in[0,1]$ and $y\in \R_+^{r_1}\times{\C^*}^{r_2}$, parametrize the line-segment $\overrightarrow{e_1 ,y}$ by $P_y(t):=(1-t)e_1+ ty$. Then there
exists $T_0=T_0(y)\in(0,1)$ such that $I_\alpha(y)=I_\alpha\big(P_y(t)\big)$ for all $\alpha\in \mathfrak{I}$ and all  $t\in[T_0,1]$. Moreover, $T_0$ can be chosen so that $P_y(t)\notin\mathcal{B}$ for all $t\in[T_0,1)$.
\end{lemm}
\begin{proof}
Suppose $\varepsilon\in I_\alpha(y)$ for some $\alpha\in \mathfrak{I}$, \ie  $\overrightarrow{e_1,\varepsilon y}$ pierces  $\overline{C}_\alpha$. By Lemma \ref{Piercee1}, $\overrightarrow{\varepsilon e_1,\varepsilon y}$ pierces  $\overline{C}_\alpha$. Thus, for some $s=s(\alpha,\varepsilon)\in(0,1)$, the point  $(1-s)\varepsilon e_1+s\varepsilon y$ is an interior point of  $\overline{C}_\alpha$. By Lemma \ref{PierceCone},
$(1-t)\varepsilon e_1+t\varepsilon y $ is also   an interior point of  $\overline{C}_\alpha$  for $s\le t<1$. But  $\varepsilon P_y(t) =(1-t)\varepsilon e_1+t\varepsilon y  $. Hence $\varepsilon\in I_\alpha\big(P_y(t)\big)$ for $s\le t\le 1$. Thus
   $I_\alpha(y)\subset I_\alpha\big(P_y(t)\big)$ for $T_0\le t\le 1$, where $T_0:= \sup_{\alpha\in \mathfrak{I},\varepsilon\in I_\alpha(y)}\{s(\alpha,\varepsilon)\}$. Since  $\mathfrak{I}$ and $I_\alpha(y)$ are  finite sets,  we have $T_0\in (0,1)$.

 We now prove the reverse inclusion, \ie   $ I_\alpha\big(P_y(t)\big) \subset  I_\alpha(y)$ for $  t\in[1-\epsilon, 1]$ for some $\epsilon>0$. Note that $P_y(t)\in \R_+^{r_1}\times{\C^*}^{r_2}$ for $t$ near 1, as $P_y(1)=y\in \R_+^{r_1}\times{\C^*}^{r_2}$. If the inclusion claimed is false, then
there is a sequence $t_j\in(0,1)$ converging to 1,   and corresponding $\varepsilon_j\in  I_\alpha\big(P_y(t_j)\big)$ with $\varepsilon_j\notin  I_\alpha(y)$.
 Thus $\varepsilon_j P_y(t_j)\in C_\alpha$ but $\varepsilon_j y\notin C_\alpha$.  By Lemma \ref{BoundedCard},  if $\Gamma\subset  \R_+^{r_1}\times{\C^*}^{r_2}$ is a  small neighborhood of $y$,   there are finitely many $\varepsilon\in E$ such that $\varepsilon \Gamma\cap C_\alpha\not=\varnothing$. Hence the set of $\varepsilon_j$ is finite.
Passing to a subsequence, we may therefore assume $\varepsilon_j=\varepsilon$ is fixed,   $\varepsilon\in I_\alpha\big(P_y(t_j)\big),\ \varepsilon\notin I_\alpha(y)$.
Thus  $\overrightarrow{e_1,\varepsilon   P_y(t_j)}$ pierces  $\overline{C}_\alpha$.  In particular, $\varepsilon   P_y(t_j)\in\overline{C}_\alpha$.
As $P_y(t_j)$ converges to $ y$ and  $\overline{C}_\alpha\cup\{0\}$ is closed in $ \R^{r_1}\times{\C}^{r_2}$, it follows that $\varepsilon y\in \overline{C}_\alpha\cup\{0\}$. Since $\varepsilon y\not=0$, it follows that $\varepsilon y\in \overline{C}_\alpha$. As   $\overrightarrow{e_1,\varepsilon   P_y(t_j)}$ pierces  $\overline{C}_\alpha$, Lemma \ref{Piercee1} shows that   $\overrightarrow{\varepsilon  e_1,\varepsilon   P_y(t_j)}$ pierces  $\overline{C}_\alpha$, \ie contains an interior point of $\overline{C}_\alpha$. But
$$
\overrightarrow{\varepsilon  e_1,\varepsilon   P_y(t_j)}=\varepsilon \cdot\Big( \overrightarrow{ e_1,   P_y(t_j)}\Big)\subset\varepsilon\cdot\Big( \overrightarrow{ e_1,   P_y(1)}\Big)=\varepsilon\cdot\big( \overrightarrow{ e_1,  y}\big)= \overrightarrow{ \varepsilon e_1, \varepsilon y}.
$$
 Hence $ \overrightarrow{ \varepsilon e_1, \varepsilon y}$  contains an interior point of $\overline{C}_\alpha$. As we have already shown that  $\varepsilon y\in \overline{C}_\alpha$,  we have proved that $ \overrightarrow{ \varepsilon e_1, \varepsilon y}$  pierces $\overline{C}_\alpha$. Hence, again by  Lemma \ref{Piercee1},   $ \overrightarrow{  e_1, \varepsilon y}$ pierces  $\overline{C}_\alpha$. Thus $\varepsilon\in I_\alpha(y)$, contradicting our choice of $\varepsilon$.

To prove the final claim in the lemma, suppose it is false. Then $P_y(t_j)\in \mathcal{B}$ for a sequence  $t_j\in(0,1)$ converging to 1. From the definition \eqref{generic} of  $\mathcal{B}$ we see that
there are $\varepsilon_j\in E$ and $\alpha_j\in \mathfrak{I}$ such that $P_y(t_j)\in\varpsilon_j f(\partial X_{\alpha_j}\times\R)$ if $\mu_{\alpha_j}\not=0$ or  $P_y(t_j)\in\varpsilon_j f( X_{\alpha_j}\times\R)$ if $\mu_{\alpha_j}=0$ . Thus $\varpsilon_j^{-1}P_y(t_j)\in \overline{C}_{\alpha_j}:=f( X_{\alpha_j}\times\R)$.
  Hence the $\varepsilon_j$ belong to a finite set. Passing to a subsequence we can assume that the
$\varepsilon_j=\varepsilon$ and $\alpha_j=\alpha$ are fixed. If $\mu_\alpha=0$, the cone generators $w_0,...,w_{n-1}$ in \eqref{fcone} are $\R$-linearly dependent. Hence
$f(\mathfrak{X}_\alpha\times\R)=\overline{C}_\alpha$ is contained in the $\R$-span $H_\alpha$ of $\ell$ elements of the number field $k$, with $\ell<n$. Thus, for  two distinct values $t_j$ we have
$P_y(t_j)\in \varpsilon H_\alpha$. In particular, the straight line connecting both of these points lies in $\varpsilon H_\alpha$. As the  $P_y(t_j)\in  \overrightarrow{e_1,\varepsilon  y}$, this line includes $e_1$. Thus $e_1\in \varpsilon H_\alpha$, contradicting Lemma \ref{e1}. If $\mu_\alpha\not=0$,
$$
f(\partial X_{\alpha}\times\R)\subset \bigcup_{i=0}^{n-1} H_{\alpha,i}, \qquad\qquad H_{\alpha,i}:=\sum_{\substack{0\le m\le n-1\\m\not=i}}\R w_m.
$$
Passing again to a subsequence of the $t_j$, we may assume $P_y(t_j)\in \varepsilon H_{\alpha,i}$ for a fixed $i$. This again contradicts Lemma \ref{e1}.
\end{proof}

We can now finish the proof of Theorem \ref{Maintheorem}, \ie the basic count  \eqref{BasicCount} for all $y\in \R_+^{r_1}\times{\C^*}^{r_2}$, as reformulated in \eqref{BasicCount2}. Using Lemma \ref{Jalpha}, we have $I_\alpha\big(P_y(t)\big)=I_\alpha(y)$ and $P_y(t)\notin\mathcal{B}$ for some $t\in(0,1)$.   Hence, by  Lemma \ref{GenericBasicCount}, 
 $$
 \sum_{\substack{\alpha\in \mathfrak{I}\\ \mu_\alpha\not=0}} \mu_\alpha  \sum_{\varepsilon\in I_\alpha(y)}  1\ =
\ \sum_{\substack{\alpha\in \mathfrak{I}\\ \mu_\alpha\not=0}}  \mu_\alpha  \sum_{\varepsilon\in I_\alpha(P_y(t))}  1\ =\ 1.
$$

\end{document}